\documentclass{amsart}
\usepackage{amssymb}
\usepackage{verbatim}
\usepackage{array}
\usepackage{latexsym}
\usepackage{enumerate}
\usepackage{tikz}
\usepackage{float}
\usepackage{amsmath}
\usepackage{amsfonts}
\usepackage{amsthm}
\usepackage{color}
\usepackage{longtable}
\usepackage[english]{babel}

\newtheorem{theorem}{Theorem}[section]
\newtheorem{proposition}[theorem]{Proposition}
\newtheorem{lemma}[theorem]{Lemma}

\def\F{\mathbb F}

\def\cH{\mathcal H}

\def\cP{\mathcal P}

\def\fqs{\mathbb{F}_{q^2}}

\def\PG{{\rm{PG}}}

\def\ord{\mbox{\rm ord}}

\def\fq{{\mathbb F}_q}

\newcommand{\PSL}{\mbox{\rm PSL}}
\newcommand{\PGL}{\mbox{\rm PGL}}

\newcommand{\PSU}{\mbox{\rm PSU}}
\newcommand{\PGU}{\mbox{\rm PGU}}

\newcommand{\aut}{\mbox{\rm Aut}}

\newcommand{\diag}{\mbox{\rm diag}}

\title{The M\"obius function of $\PSU(3,2^{2^n})$}
\date{}
\author{Giovanni Zini}

\begin{document}



\begin{abstract}
Let $G$ be the simple group ${\rm PSU}(2,2^{2^n})$, $n>0$.
For any subgroup $H$ of $G$, we compute the M\"obius function $\mu_L(H,G)$ of $H$ in the subgroup lattice $L$ of $G$, and the M\"obius function $\mu_{\bar L}([H],[G])$ of $[H]$ in the poset $\bar L$ of conjugacy classes of subgroups of $G$.
For any prime $p$, we provide the Euler characteristic of the order complex of the poset of $p$-subgroups of $G$.
\end{abstract}

\maketitle

\begin{small}




\end{small}

\section{Introduction}
The M\"obius function $\mu(H,G)$ on the subgroups of a finite group is defined recursively by $\mu(G,G)=1$ and $\sum_{K\geq H}\mu(K,G)=0$ if $H<G$. This function was used in $1936$ by Hall \cite{Hall} to enumerate $k$-tuples of elements of $G$ which generate $G$, for a given $k$.

The combinatorial and group-theoretic properties of the M\"obius function were investigated by many authors; see the paper \cite{HIO} by Hawkes, Isaacs, and \"Ozaydin. The M\"obius function is defined more generally on a locally finite poset $(\mathcal P,\leq)$ by the recursive definition $\mu(x,x)=1$, $\mu(x,y)=0$ if $x\not\leq y$, and $\sum_{x\leq z\leq y}\mu(z,y)=0$ if $x\leq y$; often, the poset taken into consideration is the subgroup lattice $L$ of a finite group $G$ ordered by inclusion.
Mann \cite{Mann,Mann2} studied $\mu(H,G)$ in the broader context of profinite groups and defined a probabilistic zeta function $P(G,s)$ associated to $G$, related to the probability of generating $G$ with $s$ elements when $G$ is positively finitely generated.

The M\"obius function on a poset $\mathcal P$ also appears in the context of topological invariants of the order simplicial complex associated to $\mathcal P$; see the works of Brown \cite{Brown1975} and Quillen \cite{Quillen}: if $\mathcal P$ is the subgroup lattice of a finite group $G$, then the reduced Euler characteristic is equal to $\mu(\{1\},G)$. This motivates the search for $\mu(\{1\},G)$ independently of the knowledge of $\mu(H,G)$ for other subgroups $H$ of $G$; see for instance \cite{Shar1,Shar2} and the references therein. Shareshian provided a formula in \cite{Shar2} for $\mu(\{1\},Sym(n))$, and computed $\mu(\{1\},G)$ in \cite{Shar1} when $G\in\{\PGL(2,q),\PSL(2,q),\PGL(3,q),\PSL(3,q),\PGU(3,q),\PSU(3,q)\}$ with $q$ odd or $G$ is a Suzuki group $Sz(2^{2h+1})$.

As already mentioned, the M\"obius function is defined for any locally finite poset.
Consider the poset $\bar L$ of conjugacy classes $[H]$ of subgroups $H$ of a finite group $G$, ordered as follows: $[H]\leq[K]$ if and only if $H$ is contained in some conjugate of $K$ in $G$.
After Hawkes, Isaacs, and \"Ozaydin \cite{HIO}, we denote by $\lambda(H,G)$ the M\"obius function $\mu([H],[G])$ in $\bar L$, while $\mu(H,G)$ is the M\"obius function in $L$.
Some attempt was done to search relations between the M\"obius functions $\mu(H,G)$ and $\lambda(H,G)$; Hawkes, Isaacs, and \"Ozaydin \cite{HIO} proved that, if $G$ is solvable, then
\begin{equation}\label{mulambda}
\mu(\{1\},G)=|G^\prime|\cdot\lambda(\{1\},G).
\end{equation}
The property \eqref{mulambda}, which we call $(\mu,\lambda)$-property, does not hold in general for non-solvable groups; see \cite{BMV}.
Pahlings \cite{Pahlings} proved that, if $G$ is solvable, then
\begin{equation}\label{generalizedmulambda}
\mu(H,G)=[N_{G^\prime}(H):H\cap G^\prime]\cdot\lambda(H,G).
\end{equation}
holds for any subgroup $H$ of $G$.
The analysis of the generalized $(\mu,\lambda)$-property \eqref{generalizedmulambda}, although false in general for non-solvable groups, is of interest since it relates the M\"obius functions $\mu(H,G)$ and $\lambda(H,G)$.

A lot of work was done by several authors  about probabilistic functions for groups; see for instance \cite{Mann,Mann2,DamianLucchini,DungLucchini}.
In particular, Mann posed in \cite{Mann} a conjecture, the validity of which would imply that the sum
$$ \sum _H \frac{\mu(H,G)}{[G:H]^s} $$
over all subgroups $H<G$ of finite index of a positively finitely generated profinite group $G$ is absolutely convergent for $s$ in some right complex half-plane and, for $s\in\mathbb{N}$ large enough, represents the probability of generating $G$ with $s$ elements.
Lucchini \cite{Lucchini2010} showed that this problem can be reduced so that Mann's conjecture is reformulated as follows:
there exist two constants $c_1,c_2\in\mathbb{N}$ such that, for any finite monolithic group $G$ with non-abelian socle,
\begin{enumerate}
\item $|\mu(H,G)|\leq[G:H]^{c_1}$ for any $H<G$ such that $G=H\, soc(G)$, and
\item the number of subgroups $H<G$ of index $n$ in $G$ such that $H\, soc(G)=G$ and $\mu(H,G)\ne0$ is upper bounded by $n^{c_2}$, for any $n\in\mathbb N$.
\end{enumerate}
It seems natural to investigate this conjecture on finite monolithic groups starting by almost simple groups.
Mann's conjecture has been shown to be satisfied by the alternating and symmetric groups \cite{CL}, as well as by those families of groups $G$ for which $\mu(H,G)$ has been computed for any subgroup $H$; namely, $\PSL(2,q)$ \cite{Hall,Downs1991}, \PGL(2,q) \cite{Downs1991}, the Suzuki groups $Sz(2^{2h+1})$ \cite{DJ}, and the Ree groups $R(3^{2h+1})$ \cite{Pierro}.

In this paper, we take into consideration the three dimensional projective special unitary group $G=\PSU(3,q)$ over the field with $q=2^{2^n}$ elements, for any positive $n$ (note that $\PSU(3,q)=\PGU(3,q)$ as $3\nmid(q+1)$). In particular, the following results are obtained.
\begin{itemize}
\item We compute $\mu(H,G)$ for any subgroup $H$ of $G$, as summarized in Table \ref{tabella1}. This shows that the groups $\PSU(3,2^{2^n})$ satisfy Mann's conjecture.
\item We compute $\lambda(H,G)$ for any subgroup $H$ of $G$, as summarized in Table \ref{tabella1}. This shows that the groups $\PSU(3,2^{2^n})$ satisfy the $(\mu,\lambda)$-property, but do not satisfy the generalized $(\mu,\lambda)$-property.
\item We compute the Euler characteristic $\chi(\Delta(L_p\setminus\{1\}))$ of the order complex of the poset of non-trivial $p$-subgroups of $G$, for any prime $p$, as summarized in Table \ref{tabella2}.
\end{itemize}
For the subgroups listed in Table \ref{tabella1}, the isomorphism type determines a unique conjugacy class in $G$.

\begin{center}
\begin{table}[!ht]
\caption{Subgroups $H$ of $G=\PSU(3,q)$, $q=2^{2^n}$, with $\mu(H)\ne0$ or $\lambda(H)\ne0$}\label{tabella1}
\def\arraystretch{1.3}
\begin{tabular}{|c|c|c|c|c|}
\hline
Isomorphism type of $H$ & $|H|$ & $N_G(H)$ & $\mu(H,G)$ & $\lambda(H,G)$ \\ \hline
$G$ & $q^3(q^3+1)(q^2-1)$ & $H$ & $1$ & $1$ \\
$(E_q\, . \, E_{q^2})\rtimes C_{q^2-1}$ & $q^3(q^2-1)$ & $H$ & $-1$ & $-1$ \\
$\PSL(2,q)\times C_{q+1}$ & $q(q^2-1)(q+1)$ & $H$ & $-1$ & $-1$ \\
$(C_{q+1}\times C_{q+1})\rtimes Sym(3)$ & $6(q+1)^2$ & $H$ & $-1$ & $-1$ \\
$C_{q^2-q+1}\rtimes C_3$ & $3(q^2-q+1)$ & $H$ & $-1$ & $-1$ \\
$E_q\rtimes C_{q^2-1}$ & $q(q^2-1)$ & $H$ & $1$ & $1$ \\
$(C_{q+1}\times C_{q+1})\rtimes C_2$ & $2(q+1)^2$ & $H$ & $1$ & $1$ \\
$Sym(3)$ & $6$ & $Sym(3)\times C_{q+1}$ & $q+1$ & $1$ \\
$C_3$ & $3$ & $C_{q^2-1}\rtimes C_2$ & $\frac{2(q^2-1)}{3}$ & $1$ \\
$C_2$ & $2$ & $(E_q\, . \, E_{q^2}) \rtimes C_{q+1}$ & $-\frac{q^3(q+1)}{2}$ & $-1$ \\
\hline
\end{tabular}
\end{table}
\end{center}


\begin{center}
\begin{table}[!ht]
\caption{Euler characteristic of the order complex of the poset of proper $p$-subgroups of $G$}\label{tabella2}
\def\arraystretch{1.5}
\begin{tabular}{|c|c c c c c|}
\hline
prime $p$ & $p\nmid|G|$ & $p=2$ & $p\mid(q+1)$ & $p\mid(q-1)$ & $p\mid(q^2-q+1)$ \\ \hline
$\chi(\Delta(L_p\setminus\{1\}))$ & $0$ & $q^3+1$ & $-\frac{q^6-2q^5-q^4+2q^3-3q^2}{3}$ & $\frac{q^6+q^3}{2}$ & $-\frac{q^6+q^5-q^4-q^3}{3}$ \\
\hline
\end{tabular}
\end{table}
\end{center}

The paper is organized as follows.
Section \ref{sec:preliminari} contains preliminary results on the M\"obius functions $\mu(H,G)$ and $\lambda(H,G)$ and the relation between the M\"obius function and the Euler characteristic of the order complex; this section contains also preliminary results on the groups $G=\PSU(3,2^{2^n})$, whose elements are described geometrically in their action on the Hermitian curve associated to $G$.
Sections \ref{sec:mobius} and \ref{sec:lambda} is devoted to the determination of $\mu(H,G)$ and $\lambda(H,G)$, respectively, for any subgroup $H$ of $G$.
Section \ref{sec:euler} provides the Euler characteristic of the order complex of the poset of proper $p$-subgroups of $G$, for any prime $p$.

\section{Preliminary results}\label{sec:preliminari}

Let $(\cP,\leq)$ be a finite poset. The M\"obius function $\mu_{\cP}:\cP\times \cP\rightarrow\mathbb{Z}$ is defined recursively as follows:
$\mu_\cP (x,y)=0$ if $x\not\leq y$; $\mu_\cP(x,x)=1$; and
$$ \mu_\cP (x,y)=0\quad\textrm{if}\quad x\not\leq y;\qquad \mu_\cP (x,x)=1;\qquad \sum_{x\leq z\leq y}\mu_\cP(z,y)=0 \quad\textrm{if}\quad x<y.$$
If $x<y$, then $\mu_\cP(x,y)$ van be equivalently defined by $\sum_{x\leq z\leq y}\mu_\cP(x,z)=0$.

To the poset $\cP$ we can associate a simplicial complex $\Delta(\cP)$ whose vertices are the elements of $\cP$ and whose $i$-dimensional faces are the chains $a_0<\cdots<a_i$ of length $i$ in $\cP$; $\Delta(\cP)$ is called the \emph{order complex} of $\cP$.
Provided that $\cP$ has a least element $0$, the Euler characteristic of the order complex of $\cP\setminus\{0\}$ is computed as follows (see \cite[Proposition 3.8.6]{Stan}):
$$ \chi(\Delta(\cP\setminus\{0\})) = -\sum_{x\in\cP\setminus\{0\}}\mu_{\cP}(0,x). $$

Given a finite group $G$, we will consider the following two M\"obius functions associated to $G$.
\begin{itemize}
\item The M\"obius function on the subgroup lattice $L$ of $G$, ordered by inclusion.
We will denote $\mu_L(H,G)$ simply by $\mu(H)$.
\item The M\"obius function on the poset $\bar L$ of conjugacy classes $[H]$ of subgroups $H$ of $G$, ordered as follows: $[H]\leq[K]$ if and only if $H$ is contained in the conjugate $gKg^{-1}$ for some $g\in G$.
We will denote $\mu_{\bar L}([H],[G])$ simply by $\lambda(H)$.
\end{itemize}
Two facts will be used to compute $\mu(H)$.
The first easy fact is that, if $H$ and $K$ are conjugate in $G$, then $\mu(H)=\mu(K)$.
The second fact is due to Hall \cite[Theorem 2.3]{Hall}, and is stated in the following lemma.
\begin{lemma}\label{intersezione}
If $H< G$ satisfies $\mu(H)\ne0$, then $H$ is the intersection of maximal subgroups of $G$.
\end{lemma}
For any prime $p$, let $L_p$ be the subposet of $L$ given by all $p$-subgroups of $G$, so that
\begin{equation}\label{caratteristica} \chi(\Delta(L_p\setminus\{1\})) = -\sum_{H\in L_p\setminus\{1\}} \mu_{L_p}(\{1\},H).
\end{equation}
By a result of Brown \cite{Brown1975}, $\chi(\Delta(L_p\setminus\{1\}))$ is congruent to $1$ modulo the order $|G|_p$ of a Sylow $p$-subgroup of $G$.
In order to compute explicitly $\chi(\Delta(L_p\setminus\{1\}))$ we will use the following result of Hall \cite[Equation (2.7)]{Hall}:
\begin{lemma}\label{pgruppi}
Let $H$ be a $p$-group of order $p^r$.
If $H$ is not elementary abelian, then $\mu_{L_p}(\{1\},H)=0$. If $H$ is elementary abelian, then $\mu_{L_p}(\{1\},H)=(-1)^r p^{\binom{r}{2}}$.
\end{lemma}

We describe now the group $G$ which will be considered in the following sections.
Let $n$ be a positive integer, $q=2^{2^n}$, $\fq$ be the finite field with $q$ element, and $\bar{\mathbb{F}}_{q}$ be the algebraic closure of $\fq$.
Let $\mathcal{U}$ be a non-degenerate unitary polarity of the plane $\PG(2,q^2)$ over $\mathbb{F}_{q^2}$, and $\cH_q\subset\PG(2,\bar{\mathbb{F}}_{q})$ be the Hermitian curve defined by $\mathcal U$.
The following homogeneous equations define models for $\cH_q$ which are projectively equivalent over $\mathbb{F}_{q^2}$:
\begin{equation} \label{M1}
X^{q+1}+Y^{q+1}+Z^{q+1}=0;
\end{equation}
\begin{equation} \label{M2}
X^{q}Z+XZ^{q}-Y^{q+1}=0.
\end{equation}
The models \eqref{M1} and \eqref{M2} are called the Fermat and the Norm-Trace model of $\cH_q$, respectively.
The set of $\mathbb{F}_{q^2}$-rational points of $\cH_q$ is denoted by $\cH_q(\fqs)$, and consists of the $q^3+1$ isotropic points of $\mathcal U$.
The full automorphism group $\aut(\cH_q)$ of $\cH_q$ is defined over $\fqs$, and coincides with the unitary subgroup $\PGU(3,q)$ of $\PGL(3,q^2)$ stabilizing $\cH_q(\fqs)$, of order $|\PGU(3,q)|=q^3(q^3+1)(q^2-1)$.

The combinatorial properties of $\cH_q(\fqs)$ can be found in \cite{HP}.
In particular, any line $\ell$ of $\PG(2,q^2)$ has either $1$ or $q+1$ common points with $\cH_q(\fqs)$, that is, $\ell$ is either a tangent line or a chord of $\cH_q(\fqs)$; in the former case $\ell$ contains its pole with respect to $\mathcal U$, in the latter case $\ell$ doesn't.
Also, $\PGU(3,q)$ acts $2$-transitively on $\cH_q(\fqs)$ and transitively on $\PG(2,q^2)\setminus\cH_q$; $\PGU(3,q)$ acts transitively also on the non-degenerate self-polar triangles $T=\{P_1,P_2,P_3\}\subset\PG(2,q^2)\setminus\cH_q$ with respect to $\mathcal U$.
Recall that, if $\sigma\in\PGU(3,q)$ stabilizes a point $P\in\PG(2,q^2)$, then $\sigma$ stabilizes also the polar line of $P$ with respect to $\mathcal U$, and viceversa.

The curve $\cH_q$ is non-singular and $\mathbb{F}_{q^2}$-maximal of genus $g=\frac{q(q-1)}{2}$, that is, the size of $\cH_q(\fqs)$ attains the Hasse-Weil upper bound $q^2+1+2gq$. This implies that $\cH_q$ is $\mathbb{F}_{q^4}$-minimal and $\mathbb{F}_{q^6}$, so that $\cH_q(\mathbb{F}_{q^4})\setminus\cH_q(\fqs)=\emptyset$ and $|\cH_q({\mathbb{F}_{q^6}})\setminus\cH_q(\fqs)|=q^6+q^5-q^4-q^3$.
Let $\Phi_{q^2}$ be the Frobenius map $(X,Y,Z)\mapsto(X^{q^2},Y^{q^2},Z^{q^2})$ over $\PG(2,\bar{\mathbb{F}}_{q^2})$; then the $\mathbb{F}_{q^6}\setminus\fqs$-rational points of $\cH_q$ split into $\frac{q^6+q^5-q^4-q^3}{3}$ non-degenerate triangles $\{P,\Phi_{q^2}(P),\Phi_{q^2}^2(P)\}$. The group $\PGU(3,q)$ is transitive on such triangles.

Since $3\nmid(q+1)$, we have $\PGU(3,q)=\PSU(3,q)$; henceforth, we denote by $G$ the simple group $\PSU(3,q)$.
The following classification of subgroups of $G$ goes back to Hartley \cite{Hartley}; here we use that $\log_2(q)$ has no odd divisors different from $1$.
The notation $S_2$ stands for a Sylow $2$-subgroup of $G$, which is a non-split extension $E_q\, . \, E_{q^2}$ of its elementary abelian center of order $q$ by an elementary abelian group of order $q^2$.

\begin{theorem}\label{Hartley} Let $n>0$, $q=2^{2^n}$, and $G=\PSU(3,q)$. Up the conjugation, the maximal subgroups of $G$ are the following.
\begin{itemize}
\item[(i)] The stabilizer $M_1(P)\cong S_2\rtimes C_{q^2-1}$ of a point $P\in\cH_q(\fqs)$, of order $q^3(q^2-1)$.
\item[(ii)] The stabilizer $M_2(P)\cong \PSL(2,q)\times C_{q+1}$ of a point $P\in\PG(2,q^2)\setminus\cH_q(\fqs)$, of order $q(q^2-1)(q+1)$.
\item[(iii)] The stabilizer $M_3(T)\cong (C_{q+1}\times C_{q+1})\rtimes Sym(3)$ of a non-degenerate self-polar triangle $T=\{P,Q,R\}\subset\PG(2,q^2)\setminus\cH_q$ with respect to $\mathcal U$, of order $6(q+1)^2$.
\item[(iv)] The stabilizer $M_4(T) \cong C_{q^2-q+1}\rtimes C_3$ of a triangle $T=\{P,\Phi_{q^2}(P),\Phi_{q^2}^2(P)\}\subset\cH_{q}(\F_{q^6})\setminus\cH_{q}(\F_{q^2})$, of order $3(q^2-q+1)$.
\end{itemize}
\end{theorem}
For a detailed description of the maximal subgroups of $G$, both from an algebraic and a geometric point of view, we refer to \cite{GSX,MZ,MZNotFixing}.

In our investigation it is useful to know the geometry of the elements of $\PGU(3,q)$ on $\PG(2,\bar{\mathbb{F}}_{q})$, and in particular on $\cH_q(\mathbb{F}_{q^2})$. This can be obtained as a corollary of Theorem \ref{Hartley}, and is stated in Lemma $2.2$ with the usual terminology of collineations of projective planes; see \cite{HP}. In particular, a linear collineation $\sigma$ of $\PG(2,\bar{\mathbb{F}}_q)$ is a $(P,\ell)$-\emph{perspectivity}, if $\sigma$ preserves each line through the point $P$ (the \emph{center} of $\sigma$), and fixes each point on the line $\ell$ (the \emph{axis} of $\sigma$). A $(P,\ell)$-perspectivity is either an \emph{elation} or a \emph{homology} according to $P\in \ell$ or $P\notin\ell$.
Lemma \ref{classificazione} was obtained in \cite{MZ} in a more general form (i.e., for any prime power $q$).
\begin{lemma}\label{classificazione}
For a nontrivial element $\sigma\in G=\PSU(3,q)$, $q=2^{2^n}$, one of the following cases holds.
\begin{itemize}
\item[(A)] ${\rm ord}(\sigma)\mid(q+1)$ and $\sigma$ is a homology, with center $P\in\PG(2,q^2)\setminus\cH_q$ and axis $\ell_P$ which is a chord of $\cH_q(\mathbb{F}_{q^2})$; $(P,\ell_P)$ is a pole-polar pair with respect to $\mathcal U$.
\item[(B)] $2\nmid{\rm ord}(\sigma)$ and $\sigma$ fixes the vertices $P_1,P_2,P_3$ of a non-degenerate triangle $T\subset\PG(2,q^6)$.
\begin{itemize}
\item[(B1)] $\ord(\sigma)\mid(q+1)$, $P_1,P_2,P_3\in\PG(2,q^2)\setminus\cH_q$, and the triangle $T$ is self-polar with respect to $\mathcal U$.
\item[(B2)] $\ord(\sigma)\mid(q^2-1)$ and $\ord(\sigma)\nmid(q+1)$; $P_1\in\PG(2,q^2)\setminus\cH_q$ and $P_2,P_3\in\cH_q(\mathbb{F}_{q^2})$.
\item[(B3)] $\ord(\sigma)\mid(q^2-q+1)$ and $P_1,P_2,P_3\in\cH_q(\mathbb{F}_{q^6})\setminus\cH_q(\mathbb{F}_{q^2})$.
\end{itemize}
\item[(C)] ${\rm ord}(\sigma)=2$; $\sigma$ is an elation with center $P\in\cH_q(\mathbb{F}_{q^2})$ and axis $\ell_P$ which is tangent to $\cH_q$ at $P$, such that $(P,\ell_P)$ is a pole-polar pair with respect to $\mathcal U$.
\item[(D)] ${\rm ord}(\sigma)=4$; $\sigma$ fixes a point $P\in\cH_q(\mathbb{F}_{q^2})$ and a line $\ell_P$ which is tangent to $\cH_q$ at $P$, such that $(P,\ell_P)$ is a pole-polar pair with respect to $\mathcal U$.
\item[(E)] $\ord(\sigma)=2d$ where $d$ is a nontrivial divisor of $q+1$; $\sigma$ fixes two points $P\in\cH_q(\mathbb{F}_{q^2})$ and $Q\in\PG(2,q^2)\setminus\cH_q$, the polar line $PQ$ of $P$, and the polar line of $Q$ which passes through $P$.
\end{itemize}
\end{lemma}
For a detailed description of the elements and subgroups of $G$, both from an algebraic and a geometric point of view, we refer to \cite{GSX,MZ,MZNotFixing}, on which our geometric arguments are based.

Throughout the paper, a nontrivial element of $G$ is said to be of type (A), (B), (B1), (B2), (B3), (C), (D), or (E), as given in Lemma \ref{classificazione}.
Also, the polar line to $\cH_q$ at $P\in\PG(2,q^2)$ is denoted by $\ell_P$.
Note that, under our assumptions, any element of order $3$ in $G$ is of type (B2).
We will denote a cyclic group of order $d$ by $C_d$ and an elementary abelian group of order $d$ by $E_d$.
The center $Z(S_2)$ of $S_2$ is elementary abelian of order $q$, and any element in $S_2\setminus Z(S_2)$ has order $4$; see \cite[Section 3]{GSX}.


\section{Determination of $\mu(H)$ for any subgroup $H$ of $G$}\label{sec:mobius}

Let $n>0$, $q=2^{2^n}$, $G=\PSU(3,q)$.
This section is devoted to the proof of the following theorem.
\begin{theorem}\label{risultato}
Let $H$ be a proper subgroup of $G$.
Then $H$ is the intersection of maximal subgroups of $G$ if and only if $H$ is one of the following groups:
\begin{equation}\label{tipologia}
\begin{array}{c}
S_2\rtimes C_{q^2-1},\quad \PSL(2,q)\times C_{q+1},\quad (C_{q+1}\times C_{q+1})\rtimes Sym(3),\quad C_{q^2-q+1}\rtimes C_3, \\
E_q \rtimes C_{q^2-1},\quad (C_{q+1}\times C_{q+1})\rtimes C_2,\quad C_{q+1}\times C_{q+1},\quad C_{q^2-1},\quad C_{2(q+1)},\\
C_{q+1}=Z(M_2(P))\textrm{ for some }P,\quad E_q,\quad Sym(3),\quad C_3,\quad C_2,\quad \{1\}.
\end{array}
\end{equation}
Given a type of groups in Equation {\rm \ref{tipologia}}, there is just one conjugacy class of subgroups of $G$ of that type.

The normalizer $N_G(H)$ of $H$ in $G$ for the groups $H$ in Equation {\rm \eqref{tipologia}} are, respectively:
\begin{equation}\label{tipologianormalizzante}
\begin{array}{c}
H,\quad H,\quad H,\quad H,\quad H,\quad H,\quad H\rtimes Sym(3),\quad H\rtimes C_2,\quad E_q\times C_{q+1},   \\
\PSL(2,q)\times H,\quad S_2\rtimes C_{q^2-1},\quad H\times C_{q+1},\quad C_{q^2-1}\rtimes C_2,\quad S_2\rtimes C_{q+1},\quad G.
\end{array}
\end{equation}
The values $\mu(H)$ for the groups $H$ in Equation {\rm \eqref{tipologia}} are, respectively:
\begin{equation}\label{tipologiamu}
\begin{array}{c}
-1,\quad -1,\quad -1,\quad -1, \quad 1,\quad 1,\quad 0,\quad 0,\quad 0,\quad 0,\quad 0,\quad q+1,\quad \frac{2(q^2-1)}{3},\quad -\frac{q^3(q+1)}{2},\quad 0.
\end{array}
\end{equation}
\end{theorem}
The proof of Theorem \ref{risultato} is divided into several propositions.
\begin{proposition}\label{unaclassediconiugio}
The group $G$ contains exactly one conjugacy class for any group in Equation {\rm \eqref{tipologia}}.
\end{proposition}
\begin{proof}
\begin{itemize}
\item The groups in the first row of Equation \eqref{tipologia} are the maximal subgroups of $G$, for which there is just one conjugacy class by Theorem \ref{Hartley}.
\item Let $\alpha_1,\alpha_2\in G$ have order $3$, so that they are of type (B2) and $\alpha_i$ fixes two distinct points $P_i,Q_i\in\cH_q(\fqs)$. The group $G$ is $2$-transitive on $\cH_q(\fqs)$, and the pointwise stabilizer of $\{P_i,Q_i\}$ is cyclic of order $q^2-1$. Hence, $\langle\alpha_1\rangle$ and $\langle\alpha_2\rangle$ are conjugated in $G$.
\item Let $\alpha_1,\alpha_2\in G$ have order $2$, so that they are of type (C) and $\alpha_i$ fixes exactly one point $P_i$ on $\cH_q(\fqs)$. Up to conjugation $P_1=P_2$, as $G$ is transitive on $\cH_q(\fqs)$. The involutions fixing $P_1$ in $G$, together with the identity, form an elementary abelian group $E_q$, which is normalized by a cyclic group $C_{q-1}$; no nontrivial element of $C_{q-1}$ commutes with any nontrivial element of $E_q$ (see \cite[Section 4]{GSX}). Hence, $\alpha_1$ and $\alpha_2$ are conjugated under an element of $C_{q-1}$.
\item Let $\alpha_1,\alpha_2,\beta_1,\beta_2 \in G$ satisfy $o(\alpha_i)=3$, $o(\beta_i)=2$, and $H_i:=\langle\alpha_i,\beta_i\rangle\cong Sym(3)$. As shown in the previous point, we can assume $\alpha_1=\alpha_2$ up to conjugation. Let $P,Q\in\cH_q(\fqs)$ and $R\in\PG(2,q^2)\setminus\cH_q$ be the fixed points of $\alpha_1$. Since $\beta_i\alpha_1\beta_i^{-1}=\alpha_1^{-1}$, we have that $\beta_i$ fixes $R$ and interchanges $P$ and $Q$; $\beta$ is then uniquely determined from the $\fqs$-rational point of $PQ$ fixed by $\beta$ (namely, the intersection between $PQ$ and the axis of $\beta$).
Since the pointwise stabilizer $C_{q^2-1}$ of $\{P,Q\}$ acts transitively on $PQ(\fqs)\setminus\cH_q$, $\beta_1$ and $\beta_2$ are conjugated, and the same holds for $H_1$ and $H_2$.
\item Any two groups isomorphic to $C_{q^2-1}$ are conjugated in $G$, because they are generated by elements of type (B2) and $G$ is $2$-transitive on $\cH_q(\fqs)$.
\item Any two groups isomorphic to $E_q$ are conjugated in $G$, because any such group fixes exactly one point $P\in\cH_q(\fqs)$, $G$ is transitive on $\cH_q(\fqs)$, and the stabilizer $G_P=M_1(P)$ contains just one subgroup $E_q$.
\item Any two groups $H_1,H_2\cong E_q\rtimes C_{q^2-1}$ are conjugated in $G$. In fact, their Sylow $2$-subgroups $E_q$ coincide up to conjugation, as shown in the previous point.
The normalizer $N_G(E_q)$ fixes the fixed point $P\in\cH_q(\fqs)$ of $E_q$, and hence $N_G(E_q)=M_1(P)=S_2\rtimes C_{q^2-1}$.
The complements $C_{q^2-1}$ are conjugated by Schur-Zassenhaus Theorem; hence, $H_1$ and $H_2$ are conjugated.
\item Any two groups isomorphic to $C_{2(q+1)}$ are conjugated in $G$, because they are generated by elements of type (E) and two elements $\alpha_1,\alpha_2$ of type (E) of the same order are conjugated in $G$.
In fact, $\alpha_i$ is uniquely determined by its fixed points $P_i\in\cH_q(\fqs)$ and $Q_i\in\ell_{P_i}(\fqs)\setminus\cH_q$; here, $\ell_{P_i}$ is the polar line of $P_i$.
Up to conjugation $P_1=P_2$, from the transitivity of $G$ on $\cH_q(\fqs)$.
Also, $S_2$ has order $q^3$ and acts on the $q^2$ points of $\ell_{P_i}(\fqs)\setminus\cH_q$ with kernel $E_q$, hence transitively.
We can then assume $Q_1=Q_2$.
\item Let $Z_{P_i}$ be the center of $M_2(P_i)$, $i=1,2$.
As shown in \cite[Section 4]{DVMZ}, $Z_{P_i}\cong C_{q+1}$ and $Z_{P_i}$ is made by the homologies with center $P_i$, together with the identity.
Since $G$ is transitive on $\PG(2,q^2)\setminus\cH_q$, we have up to conjugation that $M_2(P_1)=M_2(P_2)$ and $Z_{P_1}=Z_{P_2}$.
\item Any two groups $H_1,H_2\cong C_{q+1}\times C_{q+1}$ are conjugated in $G$. In fact, $H_i$ is the pointwise stabilizer of a self-polar triangle $T_i=\{P_i,Q_i,R_i\}\subset\PG(2,q^2)\setminus\cH_q$ (see \cite[Section 3]{DVMZ}), and the stabilizers of $T_1$ and $T_2$ are conjugated by Theorem \ref{Hartley}.
\item Any two groups $H_1,H_2\cong(C_{q+1}\times C_{q+1})\rtimes C_2$ are conjugated in $G$.
In fact, their subgroups $C_{q+1}\times C_{q+1}$ coincide up to conjugation as shown above, and fix pointwise a self-polar triangle $T=\{P,Q,R\}\subset\PG(2,q^2)\setminus\cH_q$.
Let $\beta_i\in H_i$ have order $2$, $i=1,2$. Then $\beta_i$ fixes one vertex of $T$ and interchanges the other two vertexes. Up to conjugation in $M_3(T)$ we have $\beta_1(P)=\beta_2(P)=P$.
Then $H_1=H_2$, as they coincide with the stabilizer of $P$ in $M_3(T)$.
\end{itemize}
\end{proof}

\begin{proposition}\label{normalizzante}
The normalizers $N_G(H)$ of the groups $H$ in Equation {\rm \eqref{tipologia}} are given in Equation {\rm \eqref{tipologianormalizzante}}.
\end{proposition}
\begin{proof}
\begin{itemize}
\item Clearly $N_G(H)=H$ for any $H$ in the first row of Equation \eqref{tipologia} as $H$ is maximal in $G$.
\item Let $H=E_q\rtimes C_{q^2-1}$. Then $N_G(H)\leq M_2(P)$, where $P$ is the unique fixed point of $C_{q^2-1}$ in $\PG(2,q^2)\setminus \cH_q$.
The group $H$ has a unique cyclic subgroup $C_{q+1}$ of order $q+1$; namely, $C_{q+1}$ is the center of $M_2(P)$ and is made by the homologies with center $P$; since $q$ is even, $H$ is a split extension $C_{q+1}\times (E_q\rtimes C_{q-1})$. Hence, $N_G(H)\leq N_G(C_{q+1})=M_2(P)$.
The group $H/C_{q+1}=E_q\rtimes C_{q-1}$ is self-normalizing (being maximal) in $M_2(P)/C_{q+1}=\PSL(2,q)$; thus, $N_G(E_q\rtimes C_{q-1})=H$ and $N_G(H)=H$.
\item Let $H=C_{q+1}\times C_{q+1}$. Then $N_G(H)\leq M_3(T)$, where $T$ is the self-polar triangle fixed pointwise by $H$. Since $H$ is the kernel of $M_3(T)$ in its action on $T$, we have $N_G(H)=M_3(T)$ and $|N_G(H)|=6|H|$.
\item Let $H=(C_{q+1}\times C_{q+1})\rtimes C_2$. Then $C_{q+1}\times C_{q+1}$ is normal in $N_G(H)$, being the unique subgroup of index $2$ in $H$. Hence $N_G(H)\leq M_3(T)$, where $T$ is the self-polar triangle fixed pointwise by $H$.
Also, $N_G(H)$ fixes the vertex $P$ of $T$ fixed by $H$, so that $N_G(H)\ne M_3(T)$.
This implies $N_G(H)=H$.
\item Let $H=C_{q^2-1}$. Then $H$ is generated by an element $\alpha$ of type (B2) with fixed points $P,Q\in\cH_q(\fqs)$ and $R\in\PG(2,q^2)\setminus\cH_q$.
Let $\beta$ be an involution satisfying $\beta(R)=R$, $\beta(P)=Q$, and $\beta(Q)=P$; then $\beta\in N_G(H)$, because $H$ coincides with the pointwise stabilizer of $\{P,Q\}$ in $G$.
An explicit description is the following: given $\cH_q$ with equation \eqref{M2}, we can assume up to conjugation that $\alpha=\diag(a^{q+1},a,1)$ where $a$ is a generator if $\mathbb{F}_{q^2}^*$ (see \cite{GSX}); then take
\begin{equation}\label{involuzione}
\beta=\begin{pmatrix} 0 & 0 & 1 \\ 0 & 1 & 0 \\ 1 & 0 & 0 \end{pmatrix}.
\end{equation}
Since $N_G(H)$ acts on $\{P,Q\}$ and $\beta\in N_G(H)$, the pointwise stabilizer $H$ of $\{P,Q\}$ has index $2$ in $N_G(H)$. This implies $N_G(H)=C_{q^2-1}\rtimes C_2$ and $|N_G(H)|=2|H|$.
\item Let $H=C_{2(q+1)}$, so that $H$ is generated by an element $\alpha$ of type (E) fixing exactly two points $P\in\cH_q(\fqs)$ and $Q\in\ell_P(\fqs)\setminus\cH_q$. Then $N_G(H)$ fixes $P$ and $Q$.
The subgroup $E_q$ of $M_1(P)$ commutes with $H$ elementwise, while any $2$-element in $M_1(P)\setminus E_q$ has order $4$ and does not fix $Q$; hence, the Sylow $2$-subgroup of $N_G(H)$ is $E_q$.
Also, $N_G(H)=E_q\rtimes C_d$, where $C_d$ is a subgroup of $C_{q^2-1}$ containing the subgroup $C_{q+1}$ of $H$. Let $C_2$ be the subgroup of $H$ of order $2$; the quotient group $(C_2\rtimes C_{d})/C_{q+1}\cong C_2\rtimes C_{\frac{d}{q+1}}$ acts faithfully as a subgroup of $\PGL(2,q)$ on the $q+1$ points of $\ell_Q\cap\cH_q$. By the classification of subgroups of $\PGL(2,q)$ (\cite{D}; see \cite[Hauptsatz 8.27]{Hup}), this implies $d=1$; that is, $N_G(H)=E_q\rtimes C_{q+1}$ and $|N_G(H)|=\frac{q}{2}|H|$.
\item Let $H=C_{q+1}=Z(M_2(P))$. Since $H$ is the center of $M_2(P)$, $M_2(P)\leq N_G(H)$. Conversely, $H$ is made by homologies with center $P$, and hence $N_G(H)$ fixes $P$. Thus, $N_G(H)=M_2(P)$ and $|N_G(H)|=q(q^2-1)|H|$.
\item Let $H=E_q$. Since $E_q$ has a unique fixed point $P$ on $\cH_q(\fqs)$ and $E_q=Z(M_1(P))$, we have $N_G(H)\leq M_1(P)$ and $M_1(P)\leq N_G(H)$, so that $N_G(H)=M_1(P)$ and $|N_G(H)|=q^2(q^2-1)|H|$.
\item Let $H=Sym(3)=\langle\alpha,\beta\rangle$, with $o(\alpha)=3$ and $o(\beta)=2$. Let $P,Q\in\cH_q(\fqs)$ and $R\in \PG(2,q^2)\setminus\cH_q$ be the fixed points of $\alpha$; $\beta$ fixes $R$, interchanges $P$ and $Q$, and fixes another point $A_{\beta}$ on $\ell_R\cap \cH_q$.
The group $N_G(H)$ acts on $\{P,Q\}$ and on $\{A_{\beta},A_{\alpha\beta},A_{\alpha^2\beta}\}$. The pointwise stabilizer $C_{q^2-1}$ has a subgroup $C_{q+1}$ which is the center of $M_2(P)$ and fixes $PQ$ pointwise, while any element in $C_{q^2-1}\setminus C_{q+1}$ acts semiregularly on $PQ\setminus\{P,Q\}$; hence, $C_{q^2-1}\cap N_G(H)=C_{3(q+1)}$. If an element $\gamma\in N_G(H)$ fixes $\{P,Q\}$ pointwise, then $\gamma$ fixes a point in $\{A_{\beta},A_{\alpha\beta},A_{\alpha^2\beta}\}$, and hence $\gamma\in\{\beta,\alpha\beta,\alpha^2\beta\}$.
Therefore, $N_G(H)=C_{3(q+1)}\rtimes C_2 =H\times C_{q+1}$ and $|N_G(H)|=(q+1)|H|$.
\item Let $H=C_3$ and $\alpha$ be a generator of $H$, with fixed points $P,Q\in\cH_q(\fqs)$ and $R\in\PG(2,q^2)\setminus\cH_q$. The normalizer $N_G(H)$ fixes $R$ and acts on $\{P,Q\}$.
There exists an involution $\beta\in G$ normalizing $H$ and interchanging $P$ and $Q$ (see Equation \eqref{involuzione}).
Then the pointwise stabilizer of $\{P,Q\}$ has index $2$ in $N_G(H)$.
Also, the pointwise stabilizer of $\{P,Q\}$ in $G$ is cyclic of order $q^2-1$.
Then $N_G(H)=C_{q^2-1}\rtimes C_2$ and $|N_G(H)|=\frac{2(q^2-1)}{3}|H|$.
\item Let $H=C_2$ and $\alpha$ be a generator of $H$, with fixed point $P\in\cH_q(\fqs)$. Then $N_G(H)$ fixes $P$, i.e. $N_G(H)\leq M_1(P)=S_2\rtimes C_{q^2-1}$. Since any involution of $M_1(P)$ is in the center of $S_2$, the Sylow $2$-subgroup of $N_G(H)$ has order $q^3$.
Let $\beta\in C_{q^2-1}$. If $o(\beta)\mid(q+1)$, then $\beta$ commutes with any involution of $S_2$. If $o(\beta)\nmid(q+1)$, then $\beta$ does not commute with any element of $S_2$.
This implies that $N_G(H)=S_2\rtimes C_{q+1}$, and $|N_G(H)|=\frac{q^3(q+1)}{2}|H|$.
\end{itemize}
\end{proof}

\begin{lemma}\label{involuzionefissatriangoli}
Let $\alpha\in G$ be an involution, and hence an elation, with center $P$ and axis $\ell_P$.
Then there exist exactly $q^3/2$ self-polar triangles $T_{i,j}=\{P_i,Q_{i,j},R_{i,j}\}$, $i=1,\ldots,q^2$, $j=1,\ldots,\frac{q}{2}$, such that $\alpha$ stabilizes $T_{i,j}$.
Also, $P_i\in\ell_P$ and $P\in Q_{i,j}R_{i,j}$ for any $i$ and $j$.
\end{lemma}

\begin{proof}
The number of involutions in $G$ is $(q^3+1)(q-1)$, since for any of the $q^3+1$ $\fqs$-rational points $P$ of $\cH_q$ the involutions fixing $P$ form a group $E_q$.
The number of self-polar triangles $T\subset\PG(2,q^2)\setminus\cH_q$ is $[G : M_3(T)]=\frac{(q^3+1)q^3(q^2-1)}{6(q+1)^2}$.
For any self-polar triangle $T=\{A_1,A_2,A_3\}\subset\PG(2,q^2)\setminus\cH_q$, the number of involutions in $G$ stabilizing $T$ is $3(q+1)$. In fact, for any of the $3$ vertexes of $T$ there are exactly $q+1$ involutions $\alpha_1,\ldots,\alpha_{q+1}$ fixing that vertex, say $A_1$, and interchanging $A_2$ and $A_3$; $\alpha_i$ is uniquely determined by its center $A_2 A_3 \cap \cH_q$.
Then, by double counting the size of
$$\{(\beta,T)\mid \beta\in G,\; o(\beta)=2,\; T\subset\PG(2,q^2)\setminus\cH_q\textrm{ is a self-polar triangle},\; \beta\textrm{ stabilizes }T \},$$
$\alpha$ stabilizes exactly $\frac{q^3}{2}$ self-polar triangles $T$.
For any such $T$, one vertex $P_i$ of $T$ lies on the axis of $\alpha$, because $\alpha$ is an elation, and the other two vertexes $\{Q_{i,j},R_{i,j}\}$ of $T$ lie on the polar line $\ell_{P_i}$ of $P_i$.
Since $M_1(P)$ is transitive on the $q^2$ points $P_1,\ldots,P_{q^2}$ of $\ell_P(\fqs)\setminus\{P\}$, any point $P_i$ is contained in the same number $\frac{q}{2}$ of self-polar triangles $T_{i,j}$ stabilized by $\alpha$.
\end{proof}

\begin{lemma}\label{C3fissatriangoli}
Let $\alpha\in G$ have order $3$.
Then there are exactly $\frac{q^2-1}{3}$ self-polar triangles $T_i\subset\PG(2,q^2)\setminus\cH_q$, $i=1,\ldots,\frac{q^2-1}{3}$, which are stabilized by $\alpha$.
Also, there are exactly $\frac{2(q^2-1)}{3}$ triangles $\tilde{T}_j=\{P_j,\Phi_{q^2}(P_j),\Phi_{q^2}^2 (P_j)\}\subset\cH_q(\mathbb{F}_{q^6})\setminus\cH_q(\fqs)$, $j=1,\ldots,\frac{2(q^2-1)}{3}$, which are stabilized by $\alpha$.
\end{lemma}

\begin{proof}
By Proposition \ref{unaclassediconiugio}, any two subgroups of $G$ of order $3$ are conjugated in $G$.
Also, any element of order $3$ is conjugated to its inverse by an involution of $G$. Hence, any two element of order $3$ are conjugated in $G$.

Now the claim follows by double counting the size of
$$\{(\beta,T)\mid \beta\in G,\; o(\beta)=3,\; T\subset\PG(2,q^2)\setminus\cH_q\textrm{ is a self-polar triangle},\; \beta\textrm{ stabilizes }T \},$$
$$\{(\beta,\tilde{T})\mid \beta\in G,\; o(\beta)=3,\; \tilde{T}=\{P,\Phi_{q^2}(P),\Phi_{q^2}^2 (P)\}\textrm{ for some }P\in\cH_q(\mathbb{F}_{q^6})\setminus\cH_q(\fqs),\; \beta\textrm{ stabilizes }T \},$$
using the following facts.
The number of elements of order $3$ in $G$ is $\binom{q^3+1}{2}\cdot2$.
The number of self-polar triangles $T\subset\PG(2,q^2)\setminus\cH_q$ is $[G:M_3(T)]$.
The number of elements of order $3$ stabilizing a fixed self-polar triangle $T$ is $2(q+1)^2$, because any element acting as a $3$-cycle on the vertexes of $T$ has order $3$ (see \cite[Section 3]{DVMZ}).
The number of triangles $\tilde{T}=\{P,\Phi_{q^2}(P),\Phi_{q^2}^2 (P)\}\subset\cH_q(\mathbb{F}_{q^6})\setminus\cH_q(\fqs)$ is $[G:M_4(\tilde T)]$.
The number of elements of order $3$ stabilizing a fixed triangle $\tilde T$ is $2(q^2-q+1)$, because any element in $M_4(\tilde T)\setminus C_{q^2-q+1}$ has order $3$ (see \cite[Section 4]{CKT2}).
\end{proof}

\begin{lemma}\label{S3fissatriangoli}
Let $H<G$ be isomorphic to $Sym(3)$, $H=\langle\alpha\rangle\rtimes\langle\beta\rangle$.
Then there are exactly $q+1$ self-polar triangles $T_i=\{P_i,Q_i,R_i\}\subset\PG(2,q^2)\setminus\cH_q$, $i=1,\ldots,q+1$, which are stabilized by $H$.
Up to relabeling the vertexes, we have that $P_1,\ldots,P_{q+1}$ lie on the axis of the elation $\beta$, $Q_1,\ldots,Q_{q+1}$ lie on the axis of the elation $\alpha\beta$, and $R_1,\ldots,R_{q+1}$ lie on the axis of the elation $\alpha^2\beta$.
\end{lemma}

\begin{proof}
By Proposition \ref{unaclassediconiugio}, any two subgroups $K_1,K_2<G$ with $K_i\cong Sym(3)$ are conjugated, and $|N_G(K_i)|=6(q+1)$; hence, the number of subgroups of $G$ isomorphic to $Sym(3)$ is $[G:N_G(K_i)]=\frac{(q^3+1)q^3(q-1)}{6}$.
The number of self-polar triangles $T$ is $[G:M_3(T)]=\frac{(q^2-q+1)q^3(q-1)}{6}$.
Then the claim on the number of self-polar triangles follows by double counting the size of
$$\{(K,T)\mid K< G,\; K\cong Sym(3)\; T\subset\PG(2,q^2)\setminus\cH_q\textrm{ is a self-polar triangle},\; K\textrm{ stabilizes }T \}$$
if we show that, for any self-polar triangle $T=\{A,B,C\}$, there are in $G$ exactly $(q+1)^2$ subgroups isomorphic to $Sym(3)$ which stabilize $T$.

Let $K< M_3(T)$, $K\cong Sym(3)$, $K=\langle\alpha,\beta\rangle$ with $O(\alpha)=3$, $o(\beta)=2$. Let $P,Q,R$ be the fixed points of $\alpha$, with $P\in\PG(2,q^2)\setminus\cH_q$, $Q,R\in\cH_q(\fqs)$.
By Proposition \ref{normalizzante}, $N_G(K)= K\times C_{q+1}$ where $C_{q+1}$ is made by homologies with center $P$; this implies $N_G(K)\cap M_3(T)=K$, Hence, there are at least $[M_3(T):Sym(3)]=(q+1)^2$ distinct groups $Sym(3)$ stabilizing $T$, namely the conjugates of $K$ through elements of $M_3(T)$.
On the other side, $M_3(T)$ contains exactly $(q+1)^2$ subgroups $K$ of order $3$, with fixed points $P\in\PG(2,q^2)\setminus\cH_q$, $Q,R\in\cH_q(\fqs)$. Any involution $\beta$ of $M_3(T)$ normalizing $K$ is uniquely determined by the vertex of $T$ that $\beta$ fixes, because $\beta(P)=P$, $\beta(Q)=R$, and $\beta(R)=Q$. Thus, $K$ is contained in exactly one subgroup of $M_3(T)$ isomorphic to $Sym(3)$.
Therefore the number of subgroups isomorphic to $Sym(3)$ which stabilize $T$ is $(q+1)^2$.

Finally, the configuration of the vertexes of $T_1,\ldots,T_{q+1}$ on the axes of the involutions of $H$ follows from Lemma \ref{classificazione} and the fact that every involution fixes a different vertex of $T_i$.
\end{proof}

\begin{proposition}\label{sonointersezioni}
Any group $H$ in Equation \eqref{tipologia} is the intersection of maximal subgroups of $G$.
\end{proposition}
\begin{proof}
\begin{itemize}
\item The groups in the first row of Equation \eqref{tipologia} are exactly the maximal subgroups of $G$.
\item Let $H=E_q\rtimes C_{q^2-1}$. Let $P\in\cH_q(\fqs)$ be the unique point of $\cH_q$ fixed by $E_q$; $E_q$ fixes $\ell_P$ pointwise. Also, the fixed points of $C_{q^2-1}$ are $P,Q\in\cH_q(\fqs)$ and $R\in\PG(2,q^2)\setminus\cH_q$, where $R\in\ell_P$ and $PQ=\ell_R$.
Then $H\leq M_1(P)\cap M_2(R)$.
Conversely, from $M_1(P)\cap M_2(R)\leq M_1(P)$ follows $M_1(P)\cap M_2(R)=K\rtimes C_d$ with $K\leq S_2$ and $C_d\leq C_{q^2-1}$. From $M_1(P)\cap M_2(R)\leq M_2(R)$ follows that $K$ does not contain any element of type (D), so that $K\leq E_q$. Thus, $M_1(P)\cap M_2(R)\leq H$, and $H=M_1(P)\cap M_2(R)$.
\item Let $H=(C_{q+1}\times C_{q+1})\rtimes C_2$.
Let $T=\{P,Q,R\}\subset\PG(2,q^2)\setminus\cH_q$ be the self-polar triangle fixed pointwise by $C_{q+1}\times C_{q+1}$, and let $P$ be the vertex of $T$ fixed by $C_2$. Then $H\leq M_3(T)\cap M_2(P)$.
Conversely, since $M_3(T)\cap M_2(P)$ fixes $P$ and acts on $\{Q,R\}$, the pointwise stabilizer $C_{q+1}\times C_{q+1}$ of $T$ has index at most $2$ in $M_3(T)\cap M_2(P)$, so that $M_3(T)\cap M_2(P)\leq H$.
Thus, $H=M_3(T)\cap M_2(P)$.
\item Let $H=C_{q+1}\times C_{q+1}$.
Let $T=\{P,Q,R\}\subset\PG(2,q^2)\setminus\cH_q$ be the self-polar triangle fixed pointwise by $C_{q+1}\times C_{q+1}$.
Since $H$ is the whole pointwise stabilizer of $T$ in $G$, we have $H= M_2(P)\cap M_2(Q)\cap M_2(R)$.
\item Let $H=C_{q^2-1}$ and let $\alpha$ be a generator of $H$, with fixed points $P,Q\in\cH_q(\fqs)$ and $R\in\PG(2,q^2)\setminus\cH_q$.
The pointwise stabilizer of $\{P,Q\}$ in $G$ is exactly $H$; thus, $H=M_1(P)\cap M_1(Q)$.
\item Let $H=C_{2(q+1)}$ and let $\alpha$ be a generator of $H$, of type (E), with fixed points $P\in\cH_q(\fqs)$ and $Q\in\ell_P(\fqs)\setminus\cH_q$.
By Lemma \ref{involuzionefissatriangoli} there are $\frac{q}{2}$ sel-polar triangles stabilized by the involution $\alpha^{q+1}$ having one vertex in $Q$ and two vertexes on $\ell_Q$; let $T=\{Q,R_1,R_2\}$ be one of these triangles.
Then $H\leq M_1(P)\cap M_2(Q)\cap M_3(T)$.

Conversely, let $\sigma\in (M_1(P)\cap M_2(Q)\cap M_3(T))\setminus\{1\}$. If $\sigma$ fixes $\{R_1,R_2\}$ pointwise, then from $\sigma\in M_1(P)$ follows that $\sigma$ is in the kernel $C_{q+1}\leq H$ of the action of $M_2(Q)$ on $\ell_Q$.
The quotient $(M_1(P)\cap M_2(Q)\cap M_3(T))/C_{q+1}$ acts on $\ell_Q$ as a subgroup of $\PSL(2,q)$ fixing $P$ and interchanging $R_1$ and $R_2$. From \cite[Hauptsatz 8.27]{Hup} follows $(M_1(P)\cap M_2(Q)\cap M_3(T))/C_{q+1} \cong C_2$, and hence $H=M_1(P)\cap M_2(Q)\cap M_3(T)$.
\item Let $H=C_{q+1}=Z(M_2(P))$. Then $H$ is made by the homologies of $G$ with center $P$, together with the identity.
Thus, $H=M_1(P_1)\cap M_1(P_2)\cap M_1(P_3)$, where $P_1,P_2,P_3$ are distinct point in $\ell_P \cap\cH_q$.
\item Let $H=E_q$ and let $P$ be the unique point of $\cH_q(\fqs)$ fixed by any element in $H$. Then $H=M_2(P_1)\cap M_2(P_2)\cap M_2(P_3)$, where $P_1,P_2,P_3$ are distinct points in $\ell_P(\fqs)\setminus\{P\}$.
\item Let $H=C_2$, $\alpha$ be a generator of $H$ with fixed point $P\in\cH_q(\fqs)$, and $P_1,P_2,P_3\in\ell_P(\fqs)\setminus\{P\}$.
Let $T=\{P_1,Q_{1,1},R_{1,1}\}$ be a self-polar triangle stabilized by $\alpha$.
Then $H\leq M_2(P_1)\cap M_2(P_2)\cap M_2(P_3)\cap M_3(T)$.
Since the elation $\alpha$ is uniquely determined by the image of one point not on its axis $\ell_P$, $H\leq M_3(T)$ implies $H=M_2(P_1)\cap M_2(P_2)\cap M_2(P_3)\cap M_3(T)$.
\item Let $H=C_3$. By Lemma \ref{C3fissatriangoli}, $H$ stabilizes $\frac{2(q^2-1)}{3}$ triangles $\tilde{T}\subset\cH_q(\mathbb{F}_{q^6})\setminus\cH_q(\fqs)$; let $\tilde T_1$ and $\tilde T_2$ be two of them.
Then $H\leq M_4(\tilde T_1)\cap M_4(\tilde T_2)$. If $H < M_4(\tilde T_1)\cap M_4(\tilde T_2)$, then there exist a nontrivial $\sigma\in G$ stabilizing pointwise both $\tilde T_1$ and $\tilde T_2$, a contradiction to Lemma \ref{classificazione}.
Thus, $H= M_4(\tilde T_1)\cap M_4(\tilde T_2)$.
\item Let $H=Sym(3)$. By Lemma \ref{S3fissatriangoli}, $H$ stabilizes $q+1$ self-polar triangles $T_{1},\ldots,T_{q+1}$, so that $H\leq M_3(T_1)\cap\cdots\cap M_3(T_{q+1})$. Suppose by contradiction that $H \ne M_3(T_1)\cap\cdots\cap M_3(T_{q+1})$. Then $ M_3(T_1)\cap\cdots\cap M_3(T_{q+1})$ contains a nontrivial element $\sigma$ fixing every triangle $T_i$ pointwise.
Since the triangles $T_i$'s do not have vertexes in common, this is a contradiction to Lemma \ref{classificazione}. Thus, $H = M_3(T_1)\cap\cdots\cap M_3(T_{q+1})$.
\item Let $H=\{1\}$. Since $G$ is simple, $H$ is the intersection of all maximal subgroups of $G$.
\end{itemize}
\end{proof}

\begin{proposition}\label{soloquellisonointersezioni}
If $H<G$ is the intersection of maximal subgroups, then $H$ is one of the groups in Equation \eqref{tipologia}.
\end{proposition}
\begin{proof}
We proceed as follows: we take every subgroup $K<G$ in Equation \eqref{tipologia}, starting from the maximal subgroups $M_i$ of $G$; we consider the intersections $H=K\cap M_i$ of $K$ with the maximal subgroups of $G$; here, we assume that $K\not\leq M_i$. We show that $H$ is again one of the groups in Equation \eqref{tipologia}.
\begin{itemize}
\item Let $K=S_2\rtimes C_{q^2-1}=M_1(P)$ for some $P\in\cH_q(\fqs)$.
\begin{itemize}
\item
Let $H=K\cap M_1(Q)$, $Q\ne P$. Then $H$ is the pointwise stabilizer of $\{P,Q\}\subset\cH_q(\fqs)$, which is cyclic of order $q^2-1$, i.e. $H=C_{q^2-1}$.
\item
Let $H=K\cap M_2(Q)$. Suppose $Q\in\ell_P$. Then $H=E_{q^2}\rtimes C_{q^2-1}$, where $E_{q^2}$ is made by the elations with axis $PQ$ and $C_{q^2-1}$ is generated by an element of type (B2) with fixed points $Q$, $P$, and another point $R\in\ell_Q$. 
Now suppose $Q\notin\ell_P$. Then $H$ stabilizes $\ell_Q$ and hence also the point $R=\ell_P \cap \ell_Q$. Then $H$ stabilizes $QR$ and hence also the pole $A$ of $QR$; by reciprocity, $A\in PQ$. Thus, $H$ fixes three collinear point $A,P,Q$, and hence every point on $AP$. Then $H=C_{q+1}=Z(M_2(R))$. 
\item Let $H=K\cap M_3(T)$, $T=\{A,B,C\}$, with $P$ on a side of $T$, say $P\in AB$. Then $H$ fixes $C$ and acts on $\{A,B\}$. Thus, $H$ is generated by an element of type (E) with fixed points $P,C$ and fixed lines $PC,AB$; hence, $H=C_{2(q+1)}$. 
\item Let $H=K\cap M_3(T)$, $T=\{A,B,C\}$, with $P$ out of the sides of $T$. By reciprocity, no vertex of $T$ lies on $\ell_P$. This implies that no elation acts on $T$, so that $2\nmid|H|$; this also implies that no homology in $M_3(T)$ fixes $P$, so that $H$ has no nontrivial elements fixing $T$ pointwise.
Thus $H\leq C_3$. 
\item Let $H=K\cap M_4(T)$. By Lagrange's theorem, $H\leq C_3$. 
\end{itemize}

\item Let $K=\PSL(2,q)\times C_{q+1}=M_2(P)$ for some $P\in\PG(2,q^2)\setminus\cH_q$.
\begin{itemize}
\item Let $H=K\cap M_2(Q)$, $Q\ne P$, and $R$ be the pole of $PQ$.
If $R\in PQ$, then $H$ is the pointwise stabilizer of $PQ$ and is made by the elations with center $R$; thus, $H=E_q$. 
If $R\notin PQ$, then $H$ is the pointwise stabilizer of $T=\{P,Q,R\}$; thus, $H=C_{q+1}\times C_{q+1}$. 
\item Let $H=K\cap M_3(T)$ with $T=\{A,B,C\}$.
If $P$ is a vertex of $T$, then $H=(C_{q+1}\times C_{q+1})\rtimes C_2$. 
If $P$ is on a side of $T$ but is not a vertex, say $P\in AB$, then $H$ fixes the pole $D\in AB$ of $C$. Then $H$ fixes pointwise $T^\prime=\{P,C,D\}$ and acts on $\{A,B\}$. This implies that $H$ fixes $AB$ pointwise and $H=C_{q+1}=Z(M_2(C))$. 
If $P$ is out of the sides of $T$, then no nontrivial element of $H$ fixes $T$ pointwise; thus, $H\leq Sym(3)$. 
\item Let $H=K\cap M_4(T)$. By Lagrange's theorem, $H\leq C_3$. 
\end{itemize}

\item Let $K=(C_{q+1}\times C_{q+1})\rtimes Sym(3)=M_3(T)$ for some self-polar triangle $T=\{A,B,C\}$.
\begin{itemize}
\item Let $H=K\cap M_3(T^\prime)$ with $T^\prime=\{A^\prime,B^\prime,C^\prime\}\ne T$.
If $T$ and $T^{\prime}$ have one vertex $A=A^\prime$ in common, then $H=C_{2(q+1)}$ is generated by an element of type (E) fixing $A$ and a point $D\in BC=B^\prime C^\prime$.
If $A^\prime\in AC\setminus\{A,C\}$, then $H$ stabilizes $B^\prime C^{\prime}$, because $B^\prime C^\prime$ is the only line containing $4$ points of $\{A,B,C,A^\prime,B^\prime,C^\prime\}$. Then $H$ fixes $A^\prime$, $A$, and $C$; hence also $B$. Since $H$ acts on $\{B^\prime,C^\prime\}$, $H$ cannot be made by nontrivial homologies of center $B$; thus, $H=\{1\}$.
\item Let $H=K\cap M_4(T^\prime)$. By Lagrange's theorem, $H\leq C_3$. 
\end{itemize}

\item Let $K=C_{q^2-q+1}\rtimes C_3 =M_4(T)$ for some $T\subset\cH_q(\mathbb{F}_{q^6})$.
Let $H=K\cap M_4(T^\prime)$ with $T^\prime T$. Since $3$ does not divide the order of the pointwise stabilized $C_{q^2-q+1}$ of $T$, $H$ contains no nontrivial elements fixing $T$ or $T^\prime$ pointwise. Thus, $H\leq C_3$. 

\item Let $K=E_q\rtimes C_{q^2-1}$ and $P\in\cH_q(\fqs)$, $Q\in\ell_P\setminus\{P\}$ be the fixed points of $K$.
\begin{itemize}
\item Let $H=K\cap M_1(R)$ with $R\ne P$.
If $R\in \ell_Q$, then $H=C_{q^2-1}$. 
If $R\notin\ell_Q$, then $H$ fixes the pole $S$ of $PR$; by reciprocity $S\in PQ$, so that $H$ fixes $PQ$ pointwise and also $R\notin PQ$. Thus, $H=\{1\}$.
\item Let $H=K\cap M_2(R)$ with $R\ne Q$.
If $R\in\ell_P$, then $H$ is the pointwise stabilizer $E_q$ of $PQ$. 
If $R\notin\ell_P$, then $H$ fixes pointwise the self-polar triangle $\{Q,R,S\}$ where $S$ is the pole of $QR$. Hence, either $H=C_{q+1}=Z(M_2(Q))$ or $H=\{1\}$ according to $P\in RS$ or $P\notin RS$, respectively. 
\item Let $H=K\cap M_3(T)$ with $T=\{A,B,C\}$.
If $P$ is on a side of $T$, say $P\in BC$, then either $H=\{1\}$ or $H=C_{q+1}=Z(M_2(A))$. 
If $P$ is out of the sides of $T$, then no nontrivial element of $H$ can fix $T$ pointwise; thus, $H\leq Sym(3)$. 
\item Let $H=K\cap M_4(T)$. By Lagrange's theorem, $H\leq C_3$. 
\end{itemize}

\item Let $K=(C_{q+1}\times C_{q+1})\rtimes C_2=M_3(T)\cap M_2(A)$, where $T=\{A,B,C\}$.
\begin{itemize}
\item Let $H=K\cap M_1(P)$.
If $P\in BC$, then $H=C_{2(q+1)}$ is generated by an element of type (E). 
If $P\notin BC$, then $H=\{1\}$.
\item Let $H=K\cap M_2(P)$, $P\ne A$.
If $P\in\{B,C\}$, then $H$ is the pointwise stabilizer $C_{q+1}\times C_{q+1}$ of $T$. 
If $P\in\ AB\setminus\{A,B\}$ or $P\in\ AC\setminus\{A,C\}$, then $H=C_{q+1}=Z(M_2(C))$ or $H=C_{q+1}=Z(M_2(B))$, respectively.
If $P\in BC\setminus\{B,C\}$, then $H$ fixes $A$, $P$, the pole of $AP$, and acts on $\{B,C\}$; thus, $H=C_{q+1}=Z(M_2(A))$.
If $P$ is not on the sides of $T$, then no nontrivial element of $H$ can fix $T$ pointwise; thus, $H\leq C_2$.
\item Let $H=K\cap M_3(T^{\prime})$ with $T^\prime=\{A^\prime,B^\prime,C^\prime\}\ne T$. Since $3\nmid|H|$, $H$ fixes a vertex of $T^\prime$, say $A^\prime$.
If $A^\prime =A$, then $H=C_{2(q+1)}$.
If $A^\prime\in \{B,C\}$, then $H$ fixes $T$ pointwise and acts on $\{B^\prime,C^\prime\}$; thus, $H=C_{q+1}=Z(M_2(A^\prime))$.
If $A^\prime\in (AB\cup AC)\setminus\{A,B,C\}$, then $H$ fixes $AB$ or $AC$ pointwise and acts on $\{B^\prime,C^\prime\}$; thus, $H=\{1\}$.
If $A^\prime\in BC$, then $H$ fixes $A$, $A^\prime$, and the pole $D$ of $A A^\prime$; as $H$ acts on $\{B,C\}$, this implies $H=\{1\}$.
If $A^\prime$ is not on the sides of $T$, then no nontrivial element of $H$ fixes $T$ pointwise and $H\leq C_2$.
\item Let $H=K\cap M_4(T^\prime)$. By Lagrange's theorem, $H\leq C_3$.
\end{itemize}

\item Let $K=C_{q+1}\times C_{q+1}=M_3(T)\cap M_2(A)\cap M_2(B)\cap M_2(C)$ with $T=\{A,B,C\}$.
\begin{itemize}
\item Let $H=K\cap M_1(P)$ or $H=K\cap M_2(P)$. If $P$ is not on the sides of $T$, then $H=\{1\}$; if $P$ is on a side of $T$, say $P\in BC$, then $H=C_{q+1}=Z(M_2(A))$.
\item Let $H=K\cap M_3(T^\prime)$ with $T^\prime=\{A^\prime,B^\prime,C^\prime\}$. Since $K$ is not divisible by $2$ or $3$, $H\ne\{1\}$ only if $H$ fixes $T^\prime$ pointwise. Up to relabeling, this implies $A^\prime =A$, $B^\prime, C^\prime \in BC$, and $H=C_{q+1}=Z(M_2(A))$.
\item Let $H=K\cap M_4(T^\prime)$. By Lagrange's theorem, $H=\{1\}$.
\end{itemize}

\item Let $K=C_{q^2-1}=\langle\alpha\rangle$, with $\alpha$ of type (B2) fixing the points $P\in\PG(2,q^2)\setminus\cH_q$ and $Q,R\in\cH_q(\fqs)$.
\begin{itemize}
\item Let $H=K\cap M_1(A)$ or $H=K\cap M_2(A)$. Since the nontrivial elements of $H$ are either of type (B2) or of type (A) with axis $QR$, we have $H=\{1\}$ unless $A\in QR$; in this case, $H=C_{q+1}=Z(M_2(P))$.
\item Let $H=K\cap M_3(T)$ or $H=K\cap M_4(T)$. By Lagrange's theorem, $H\leq C_3$.
\end{itemize}

\item Let $K=C_{2(q+1)}=\langle\alpha\rangle$ with $\alpha$ of type (E) fixing the points $P\in\cH_q(\fqs)$ and $Q\in\PG(2,q^2)\setminus\cH_q$.
\begin{itemize}
\item Let $H=K\cap M_1(R)$ or $H=K\cap M_2(R)$.
If $R\in\ell_Q$, then $H=C_{q+1}=Z(M_2(Q))$. If $R\notin \ell_Q$, then $H=\{1\}$.
\item Let $H=K\cap M_3(T)$; recall that $H<K$.
If $Q$ is a vertex of $T$, then $H=C_{q+1}=Z(M_2(Q))$.
If $Q$ is not a vertex of $T$, then no homology in $K$ acts on $T$; hence, $H\leq C_2$.
\item Let $H=K\cap M_4(T)$. By Lagrange's theorem, $H=\{1\}$.
\end{itemize}

\item Let $K=C_{q+1}=Z(M_2(P))$ for some $P\in \PG(2,q^2)\setminus\cH_q$ and $\sigma\in K\setminus\{1\}$.
Then $\sigma$ fixes no points out of $\{P\}\cup\ell_P$; also, the triangles fixed by $\sigma$ have one vertex in $P$ and two vertexes on $\ell_P$.
Thus, $K\cap M_i=\{1\}$ for amy maximal subgroup $M_i$ of $G$ not containing $K$.

\item Let $K=E_q$ and $\sigma\in E_q\setminus\{1\}$. Recall that $K$ fixes one point $P\in\cH_q(\fqs)$ and the line $\ell_P$ pointwise.
Also, $\sigma$ fixes no points out of $\ell_P$. If $\sigma$ fixes a triangle $T=\{A,B,C\}$, then one vertex of $T$ lies on $\ell_P(\fqs)$, say $A$, and $\sigma$ is uniquely determined by $\sigma(B)=C$.
Thus, $K\cap M_1(Q)=K\cap M_2(Q)=K\cap M_4(T^\prime)=\{1\}$ and $K\cap M_3(T)\leq C_2$.

\item Let $K\in\{Sym(3),C_3,C_2,\{1\}\}$. Then every subgroup of $K$ is in Equation \eqref{tipologia}.
\end{itemize}
\end{proof}

\begin{proposition}\label{valoridimu}
The values $\mu(H)$ for the groups in Equation \eqref{tipologia} are given in Equation \eqref{tipologiamu}.
\end{proposition}

\begin{proof}
Let $H$ be one of the groups in Equation \eqref{tipologia}. By Lemma \ref{intersezione} and Proposition \ref{soloquellisonointersezioni}, $\mu(H)$ only depends on the subgroups $K$ of $G$ such that $H<K$ and $K$ is in Equation \eqref{tipologia}.

\begin{itemize}
\item If $H$ is in the first row of Equation \eqref{tipologia}, then $H$ is maximal in $G$, and hence $\mu(H)=-1$.

\item Let $H=E_q\rtimes C_{q^2-1}$. Let $P\in\cH_q(\fqs)$ and $Q\in\PG(2,q^2)\setminus\cH_q$ be the fixed points of $H$.
Then $H=M_1(P)\cap M_2(Q)$ and $H$ is not contained in any other maximal subgroup of $G$. Thus, $\mu(H)=-\{\mu(G)+\mu(M_1(P))+\mu(M_2(Q))\}=1$.

\item Let $H=(C_{q+1}\times C_{q+1})\rtimes C_2$.
Let $T=\{P,Q,R\}$ be the self-polar triangle stabilized by $H$, with $H(P)=P$.
No point different from $P$ is fixed by $H$. Also, if a triangle $T^\prime=\{P^\prime,Q^\prime\} \ne T$ is fixed by $H$, then $P$ is a vertex of $T^\prime$, say $P=P^\prime$, and $\{Q^\prime,R^\prime\}\subset QR$; but $C_{q+1}\times C_{q+1}$ has orbits of length $q+1>|\{Q^\prime,R^\prime\}|$, so that $H$ cannot fix $T^\prime$.
Then $H=M_2(P)\cap M_3(T)$ and $H$ is not contained in any other maximal subgroup of $G$.
Thus, $\mu(H)=1$.

\item Let $H=C_{q+1}\times C_{q+1}$ and $T=\{P,Q,R\}$ be the self-polar triangle fixed pointwise by $H$.
The vertexes of $T$ are the unique fixed points of the elements of type (B1) in $H$.
Also, any triangle $T^\prime\ne T$ fixed by an element of type (A) in $H$ has two vertexes on a side $\ell$ of $T$; but $H$ has orbits of length $q+1>2$ on $\ell$, so that $H$ does not fix $T^\prime$.
Then $H=M_3(T)\cap M_2(P)\cap M_2(Q)\cap M_2(R)$ and $H$ is not contained in any other maximal subgroup of $G$.

If $K$ is one of the groups $M_3(T)\cap M_2(P)$, $M_3(T)\cap M_2(P)$, $M_3(T)\cap M_2(P)$, then $K$ contains $H$ properly, and $\mu(K)=1$ as shown in the previous point.
The intersection of three groups between $M_3(T)$, $M_2(P)$, $M_2(Q)$, and $M_2(R)$ is equal to $H$.
Thus, by direct computation, $\mu(H)=0$.

\item Let $H=C_{q^2-1}$ with fixed points $P\in\PG(2,q^2)\setminus\cH_q$ and $Q,R\in\cH_q(\fqs)$. Then $H=M_1(Q)\cap M_1(R) = M_1(Q)\cap M_1(R)\cap M_2(P)$.
We already know $\mu(M_1(Q)\cap M_2(P))=\mu(M_1(R)\cap M_2(P))=1$.
Moreover, $C_{q^2-1}$ has no fixed triangles, by Lagrange's theorem, and no other fixed points.
Thus, by direct computation, $\mu(H)=0$.

\item Let $H=C_{2(q+1)}=\langle\alpha\rangle$; $\alpha$ is of type (E), fixes the points $P\in\cH_q(\fqs)$ and $Q\in PG(2,q^2)\setminus\cH_q$, and fixes the lines $\ell_P$ and $\ell_Q$.
Since $\alpha^2$ is a homology with center $Q$, the orbits on $\ell_Q$ of $H$ coincide with the orbits on $\ell_Q$ of the elation $\alpha^{q+1}$.
By Lemma \ref{involuzionefissatriangoli}, the self-polar triangles $T_i$ stabilized by $H$ have a vertex in $Q$ and two vertexes on $\ell_Q$; there are exactly $\frac{q}{2}$ such triangles $T_1,\ldots,T_{\frac{q}{2}}$. No other triangle and no other point different from $P$ and $Q$ is fixed by $H$, so that $H=M_1(P)\cap M_2(Q)\cap M_3(T_1)\cap\cdots\cap M_3(T_{\frac{q}{2}})$ and $H$ is not contained in any other maximal subgroup of $G$.

If $K$ is the intersection of $M_2(Q)$ with one of the groups $M_1(P),M_3(T_1),\ldots, M_3(T_{\frac{q}{2}})$, then $K=E_q\rtimes C_{q^2-1}$ or $K=(C_{q+1}\times C_{q+1})\rtimes C_2$; hence, $K$ contains $H$ properly and $\mu(K)=1$ as shown above.
The intersection of $K$ with a third maximal subgroup of $G$ containing $H$ coincides with $H$.
Finally, the intersection of any two groups in $\{M_1(P),M_3(T_1),\ldots,M_3(T_{\frac{q}{2}})\}$ coincides with $H$.
Thus, by direct computation, $\mu(H)=0$.

\item Let $H=C_{q+1}=Z(M_2(P))$. Denote $\ell_P\cap\cH_q=\{P_1,\ldots,P_{q+1}\}$ and $\ell(\fqs)\setminus\cH_q=\{Q_1,\ldots,Q_{q^2-q}\}$ such that, for $i=1,\ldots,\frac{q^2-q}{2}$, $T_{i}=\{P,Q_i,Q_{i+\frac{q^2-q}{2}}\}$ are the self-polar triangles with a vertex in $P$.
Then
$$H=M_1(P_1)\cap\cdots\cap M_1(P_{q+1})\cap M_2(P)\cap M_2(Q_1)\cap\cdots\cap M_2(Q_{q^2-q})\cap M_3(T_1)\cap\cdots\cap M_3(T_{\frac{q^2-q}{2}})$$
and $H$ is not contained in any other maximal subgroup of $G$.
By direct inspection, the intersections $K$ of some (at least two) maximal subgroups of $G$ such that $H<K<G$ are exactly the following.
\begin{itemize}
\item $K=M_1(P_i)\cap M_1(P_j)$ for some $i\ne j$; in this case, $K=C_{q^2-1}$ and $\mu(K)=0$.
\item $K=M_1(P_i)\cap M_2(P)$ with $i\in\{1,\ldots,q+1\}$; in this case, $K=E_q\rtimes C_{q^2-1}$ and $\mu(K)=1$. These $q+1$ groups are pairwise distinct.
\item $K=M_1(P_i)\cap M_3(T_j)$ for some $i,j$; in this case, $K=C_{2(q+1)}$ and $\mu(K)=0$.
\item $K=M_2(P)\cap M_2(Q_i)$ for some $i$; in this case, $K=C_{q+1}\times C_{q+1}$ and $\mu(K)=0$.
\item $K=M_2(P)\cap M_3(T_i)$ with $i\in\{1,\ldots,\frac{q^2-q}{2}\}$; in this case, $K=(C_{q+1}\times C_{q+1})\rtimes C_2$ and $\mu(K)=1$. These $\frac{q^2-q}{2}$ groups are pairwise distinct.
\item $K=M_2(Q_i)\cap M_3(T_i)$ or $K=M_2(Q_{i+\frac{q^2-q}{2}})\cap M_3(T_i)$, with $i\in\{1,\ldots,\frac{q^2-q}{2}\}$; in  this case, $K=(C_{q+1}\times C_{q+1})\rtimes C_2$ and $\mu(K)=0$.
These $q^2-q$ groups are pairwise distinct.
\end{itemize}
To sum up, the only subgroups $K$ with $H<K<G$ and $\mu(K)\ne0$ are the maximal subgroups, $q+1$ distinct groups of type $E_q \rtimes C_{q^2-1}$, and $\frac{3(q^2-q)}{2}$ distinct groups of type $(C_{q+1}\times C_{q+1})\rtimes C_2$.
Thus, $\mu(H)=0$.

\item Let $H=E_q$. Let $P$ be the point of $\cH_q(\fqs)$ fixed by $H$; $H$ fixes $\ell_P$ pointwise.
We have $H=M_1(P)\cap M_2(Q_1)\cap\cdots\cap M_2(Q_{q^2})$, where $Q_1,\ldots,Q_{q^2}$ are the $\fqs$-rational points of $\ell_P\setminus\{P\}$; $H$ is not contained in any other maximal subgroup of $G$.
The intersections $K$ of at least two maximal subgroups of $G$ such that $H<K<G$ are exactly the $q^2$ groups $M_1(P)\cap M_2(Q_i)=E_q\rtimes C_{q^2-1}$, with $\mu(K)=1$.
Thus, by direct computation, $\mu(H)=0$.

\item Let $H=Sym(3)=\langle\alpha,\beta\rangle$ with $o(\alpha)=3$ and $o(\beta)=2$. Let $P\in\PG(2,q^2)\setminus\cH_q$ and $Q,R\in\cH_q$ be the fixed points of $\alpha$, and $A\in QR$ be the fixed point of $\beta$ on $\cH_q$, so that $\beta$ fixes $\ell_A=AP$.
By Lemma \ref{S3fissatriangoli} and its proof, $H=M_2(P)\cap M_3(T_1)\cap\cdots\cap M_3(T_{q+1})$, where $T_i$ has one vertex on $\ell_A\setminus\{P,A\}$ and the other two vertexes are collinear with $A$; $H$ is not contained in any other maximal subgroup of $G$.

For any $i,j\in\{1,\ldots,q+1\}$ with $i\ne j$, no vertex of $T_j$ is on a side of $T_i$; hence, no nontrivial element of $M_3(T_i)\cap M_3(T_j)$ fixes $T_i$ pointwise. This implies $M_3(T_i)\cap M_3(T_j)=H$.
Analogously, no nontrivial element in $M_3(T_i)\cap M_2(P)$ fixes $T_i$ pointwise, and this implies $M_3(T_i)\cap M_2(P)=H$.
Thus, by direct computation, $\mu(H)=q+1$.

\item Let $H=C_3=\langle\alpha\rangle$ with fixed points $P\in\PG(2,q^2)\setminus\cH_q$ and $Q,R\in\cH_q$.
By Lemma \ref{C3fissatriangoli},
$$H=M_1(Q)\cap M_1(R)\cap M_2(P)\cap M_3(T_1)\cap\cdots\cap M_3(T_\frac{q^2-1}{3})\cap M_4(\tilde{T}_1)\cap\cdots\cap M_4(\tilde{T}_{\frac{2(q^2-1)}{3}}) $$ 
and $H$ is not contained in  any other maximal subgroup of $G$.
By direct inspection, the intersections $K$ of at least two maximal subgroups of $G$ such that $H<K<G$ are exactly the following.
\begin{itemize}
\item $K=M_1(Q)\cap M_2(P)$ or $K=M_1(R)\cap M_2(P)$; in this case, $K=E_q\rtimes C_{q^2-1}$ and $\mu(K)=1$.
\item $K=M_1(Q)\cap M_1(R)$; in this case, $K=C_{q^2-1}$ and $\mu(K)=0$.
\item There are exactly $\frac{q-1}{3}$ groups $K$ containing $H$ with $K\cong Sym(3)$, and hence $\mu(K)=q+1$.
In fact, any involution $\beta\in G$ satisfying $\langle H,\beta\rangle\cong Sym(3)$ interchanges $Q$ and $R$ and fixes a point of $(QR\cap\cH_q)\setminus\{P,Q\}$; conversely, any of the $q-1$ points $A_1,\ldots,A_{q-1}$ of $(QR\cap\cH_q)\setminus\{P,Q\}$ determines uniquely the involution $\beta_i\in G$ such that $\beta(A_i)$, $\beta_i(Q)=R$, $\beta_i(R)=Q$, and hence $\langle H,\beta_i\rangle\cong Sym(3)$.
The involutions $\beta_i$, $\alpha\beta_i$, and $\alpha^2\beta_i$, together with $H$, generate the same group; thus, there are exactly $\frac{q-1}{3}$ groups $Sym(3)$ containing $H$.
\end{itemize}
Thus, by direct computation, $\mu(H)=\frac{2(q^2-1)}{3}$.

\item Let $H=C_2=\langle\alpha\rangle$, where $\alpha$ has center $P$. Let $\ell_P(\fqs)\setminus\{P\}=\{P_1,\ldots,P_{q^2}\}$. By Lemma \ref{involuzionefissatriangoli},
$$ H= M_1(P)\cap \bigcap_{i=1}^{q^2} M_2(P_i)\cap \bigcap_{i=1}^{q^2}\bigcap_{j=1}^{q/2} M_3(T_{i,j}),  $$
where the triangles $T_{i,j}$ are described in Lemma \ref{involuzionefissatriangoli}; $H$ is not contained in any other maximal subgroup of $G$.
By direct inspection, the intersections $K$ of at least two maximal subgroups of $G$ such that $H<K<G$ are exactly the following.
\begin{itemize}
\item $K=M_1(P)\cap M_2(P_i)$ for $i=1,\ldots,q^2$; in this case, $K=E_q\rtimes C_{q^2-1}$ and $\mu(K)=0$.
\item $K=M_2(P_i)\cap M_2(P_j)$ with $i\ne j$; in this case, $K=E_q$ and $\mu(K)=0$.
\item $K=M_1(P)\cap M_3(T_{i,j})$; in this case, $K=E_q\rtimes C_{2(q+1)}$ and $\mu(K)=0$.
\item $K=M_2(Q_i)\cap M_3(T_{i,j})$ with $i\in\{1,\ldots,q^2\}$ and $j\in\{1,\ldots,\frac{q}{2}\}$; these $\frac{q^3}{2}$ distinct groups are of type $(C_{q+1}\times C_{q+1})\rtimes C_2$, so that $\mu(K)=1$.
\item There are exactly $N=\frac{q^3}{2}$ groups $K$ containing $H$ such that $K\cong Sym(3)$, and hence $\mu(K)=q+1$.
This follows by double counting the size of
$$ I=\{(H,K)\mid\; H,K<G,\;H\cong C_2,\;K\cong Sym(3),\; H<K \}. $$
Arguing as in the proof of Lemma \ref{involuzione}, $|I|=(q^3+1)(q-1)N$; arguing as in the proof of Lemma \ref{S3fissatriangoli}, $|I|=\frac{q^3(q^3+1)(q-1)}{6}\cdot3$.
Hence, $N=\frac{q^3}{2}$.
\end{itemize}
Thus, by direct computation, $\mu(H)=-\frac{q^3(q+1)}{2}$.

\item Let $H=\{1\}$. Then $\mu(H)=-\sum_{\{1\}<K\leq G} \mu(K,G)$. By the values $\mu(K)$ computed in the previous cases, Propositions \ref{unaclassediconiugio}, and Proposition \ref{normalizzante}, only the following groups $K$ have to be considered:
$1$ group $G$; $q^3+1$ groups $S_2\rtimes C_{q^2-1}$; $q^2(q^2-q+1)$ groups $\PSL(2,q)\times C_{q+1}$; $\frac{q^3(q-1)(q^2-q+1)}{6}$ groups $(C_{q+1}\times C_{q+1})\rtimes Sym(3)$; $\frac{q^3(q+1)^2(q-1)}{3}$ groups $C_{q^2-q+1}\rtimes C_3$; $(q^3+1)q^2$ groups $E_q\rtimes C_{q^2-1}$; $\frac{q^3(q-1)(q^2-q+1)}{2}$ groups $(C_{q+1}\times C_{q+1})\rtimes C_2$; $\frac{q^3(q^3+1)(q-1)}{6}$ groups $Sym(3)$; $\frac{q^3(q^3+1)}{2}$ groups $C_3$; $(q^3+1)(q-1)$ groups $C_2$.
Thus, by direct computation, $\mu(H)=0$.
\end{itemize}
\end{proof}

\section{Determination of $\lambda(H)$ for any subgroup $H$ of $G$}\label{sec:lambda}

Let $n>0$, $q=2^{2^n}$, $G=\PSU(3,q)$. This section is devoted to the proof of the following theorem.

\begin{theorem}\label{risultatolambda}
Let $H$ be a proper subgroup of $G$.
Then $\lambda(H)\ne0$ if and only $H$ is one of the following groups:
\begin{equation}\label{tipologialambda}
\begin{array}{c}
E_q\rtimes C_{q^2-1},\quad (C_{q+1}\times C_{q+1})\rtimes C_2,\quad Sym(3),\quad C_3, \\
S_2\rtimes C_{q^2-1},\quad \PSL(2,q)\times C_{q+1},\quad (C_{q+1}\times C_{q+1})\rtimes Sym(3),\quad C_{q^2-q+1}\rtimes C_3,\quad C_2.
\end{array}
\end{equation}
For any isomorphism type in Equation \eqref{tipologialambda} there is just one conjugacy class of subgroups of $G$.

If $H$ is in the first row of Equation \eqref{tipologialambda}, then $\lambda(H)=-1$; if $H$ is in the second row of Equation \eqref{tipologialambda}, then $\lambda(H)=1$.
\end{theorem}

\begin{proof}
By Proposition \ref{unaclassediconiugio}, for any isomorphism type in Equation \eqref{tipologialambda} there is just one conjugacy class of subgroups of $G$ of that type.
Hence, we can use the notation $[M_1]$, $[M_2]$, $[M_3]$ and $[M_4]$ for the conjugacy classes of $M_1(P)$, $M_2(P)$, $M_3(T)$ and $M_4(T)$, respectively.
If $H=G$, then $\lambda(H)=1$; if $H$ is one of the groups in the second row of Equation \eqref{tipologialambda} and $H\ne C_2$, then $\lambda(H)=-1$ as $H$ is maximal in $G$.

Firstly, we assume that $H$ is not a subgroup of $Sym(3)$, and that $H$ is not a group of homologies, i.e. $H\not\leq C_{q+1}=Z(M_2(Q))$ for any point $Q$.
\begin{itemize}
\item Let $H< M_4(T)$ for some $T$. From $H\ne C_3$ follows that some nontrivial element in $H$ fixes $T$ pointwise; hence, $H$ is not contained in any maximal subgroup of $G$ other than $M_4(T)$. Thus, inductively, $\lambda(H)=-\{\lambda(G)+\lambda(M_4(T))\}=0$.

\item Let $H< M_1(P)$ for some $P$; we assume in addition that $\gcd(|H|,q-1)>1$. Here, the assumption $H\not\leq Sym(3)$ reads $H\notin\{\{1\},C_2,C_3\}$.
If $H$ contains an element of order $4$, then $H$ is not contained in any maximal subgroup of $G$ other than $M_1(P)$. Thus, inductively, $\lambda(H)=0$.

We can then assume that the $2$-elements of $H$ are involutions, so that $H=E_{2^r}\rtimes C_{d}$ with $0\leq r\leq 2^n$ and $d\mid(q^2-1)$ (see \cite[Theorem 11.49]{HKT}).
This implies that $H\leq M_1(P)\cap M_2(Q)$ for some $Q\in\ell_P$; the eventual nontrivial elements in $H$ whose order divides $q+1$ are homologies with center $Q$.
Then we have $[H]\leq [M_1]$, $[H]\leq [M_2]$; by Lagrange's theorem, $[H]\not\leq[M_4]$.
From the assumptions $\gcd(|H|,q-1)>1$ and $H\not\leq Sym(3)$ follows $[H]\not\leq[M_3]$.

If $H=E_q\rtimes C_{q^2-1}$, then no proper subgroup of $M_1(P)$ or $M_2(Q)$ contains $H$ properly; thus, $\lambda(H)=1$.
If $H\ne E_q\rtimes C_{q^2-1}$, then $H< E_q\rtimes C_{q^2-1}=M_1(P)\cap M_2(Q)$ up to conjugation. Thus, inductively, the only classes $[K]$ with $[H]\leq [K]$ and $\lambda(K)\ne0$ are $[K]\in\{[G],[M_1],[M_2],[E_q\rtimes C_{q^2-1}]\}$. This implies $\lambda(H)=0$.

\item Let $H<M_2(Q)$ for some $Q$, and assume also $H\not\leq M_1(P)$ for any $P$. As $H\not\leq C_3$, we have $[H]\not\leq [M_4]$. The group $\bar H:= H/(H\cap Z(M_2(Q)))$ acts as a subgroup of $\PSL(2,q)$ on $\ell_Q\cap\cH_q$; we assume in this point that $H$ is one of the following groups (see \cite[Hauptsatz 8.27]{Hup}): $\PSL(2,2^{2^h})$ with $0<h\leq n$; a dihedral group of order $2d$ where $d$ is a divisor of $q-1$ greater than $3$; $Alt(5)$.
Then, by Lagrange's theorem, $[H]\not\leq[M_3]$.
Thus, inductively, $G$ and $M_2(Q)$ are the only groups $K$ with $H<K$ and $\lambda(K)\ne0$, so that $\lambda(H)=0$.

Note that, since we are under the assumptions $H\not\leq M_1(P)$ for any $P$, $H\not\leq Sym(3)$, and $H\not\leq C_{q+1}=Z(M_2(Q))$, we have that the only subgroups $\bar H$ of $\PSL(2,q)$ for which $\lambda(H)$ still has not been computed are the cyclic or dihedral groups of order $d$ or $2d$ (respectively), where $d$ is a nontrivial divisor of $q+1$.

\item Let $H< M_3(T)$ for some $T$, and assume also $H\not\leq M_1(P)$ for any $P$. As $H\not\leq C_3$, we have $[H]\not\leq [M_4]$.
Here, the assumption $H\not\leq Sym(3)$ means that some nontrivial element of $H$ fixes $T$ pointwise. Hence, the assumption $H\not\leq C_{q+1}=Z(M_2(Q))$ for any vertex $Q$ of $T$, together with $H\not\leq M_1(P)$, implies that $H$ contains some element of type (B1).
Write $H=L\rtimes K$, with $K\leq Sym(3)$ and $L<C_{q+1}\times C_{q+1}$.

If $K=C_3$ or $K=Sym(3)$, then $[H]\not\leq[M_2]$; thus, inductively, $G$ and $M_3(T)$ are the only groups $K$ with $H<K$ and $\lambda(K)\ne0$, so that $\lambda(H)=0$.

If $K=C_2$ and $L=C_{q+1}\times C_{q+1}$, then $H\leq M_2(Q)$ for some vertex $Q$ of $T$. Since $\bar H:= H/(H\cap Z(M_2(Q)))$ is dihedral of order $2(q+1)$, \cite[Haptsatz 8.27]{Hup} implies the non-existence of groups $K$ with $H<K<M_2(Q)$ (except for $q=4$ and $\bar{K}=Alt(5)$; in this case, $\lambda(K)=0$ by the previous point).
Thus, $\lambda(H)=-\{\lambda(G)+\lambda(M_2(Q))+\lambda(M_3(T))\}=1$.

If $K=C_2$ and $L<C_{q+1}\times C_{q+1}$, then again $H\leq M_2(Q)$ with $Q$ vertex of $T$. The group $\bar H$ is dihedral of order $2d$, where $d\mid(q+1)$; $d>1$ because $L$ contains elements of type (B1).
By the previous point and \cite[Hauptsatz 8.27]{Hup}, the only groups $K$ with $H<K<M_2(Q)$ are such that $\bar K$ is dihedral of order dividing $q+1$.
Thus, inductively, $\lambda(H)=0$.

If $K=\{1\}$, then $H\in M_2(Q)$ for any vertex $Q$ of $T$.
The group $\bar H<\PSL(2,q)$ on the line $\ell_Q\cap\cH_q$ is cyclic of order $d\mid(q+1)$; $d>1$ because $H$ has elements of type (B1).
By \cite[Hauptsatz 8.27]{Hup}, the groups $K$ with $H<K< M_2(Q)$ are such that either $\bar K$ is cyclic of order dividing $q+1$, or we have already proved that $\lambda(K)=0$.
Thus, inductively, $\lambda(K)=0$.

\item Let $H<M_2(Q)$ for some $Q$. Let $\bar H\ne\{1\}$ be the induced subgroup of $\PSL(2,q)$ acting on $\ell_Q\cap\cH_q$.
If $\bar H$ is cyclic or dihedral of order $d$ or $2d$ (respectively) with $d\mid(q+1)$, then $H\leq M_3(T)$ for some $T$. Hence, $\lambda(H)=0$, as already computed in the previous point in the case $K=\{1\}$ if $\bar H$ is cyclic, or in the case $K=C_2$ if $H$ is dihedral.

\item Under the assumptions that $H\not\leq Sym(3)$ and $H$ is not a group of homologies, the only remaining case is $H<M_1(P)$ for some $P$ with $\gcd(|H|,q-1)=1$.
In this case $H=E_{2^r}\times C_d$, where $C_d$ is cyclic of order $d\mid(q+1)$ and made by homologies, whose axis passes through $P$ and whose center $Q$ lies on $\ell_P$.
We have $r>0$, because $H\not\leq Z(M_2(Q))$.

If $r=1$, then $H$ is cyclic of order $2d$ generated by an element of type (E). By Lemma \ref{involuzionefissatriangoli}, $H\leq M_3(T)$, where $T$ has a vertex in $Q$ and two vertexes on $\ell_Q$.
Hence, $[H]\leq[M_1]$, $[H]\leq[M_2]$, $[H]\leq[M_3]$, and $[H]\not\leq[M_4]$.
Let $K$ be such that $H<K\leq G$ and $K$ is not of the same type of $H$, i.e. $K$ is not cyclic of order $2d^\prime$ with $d^\prime\mid(q+1)$.
As shown in the previous points, $\lambda(K)\ne0$ if and only if $[K]\in\{[G],[M_1],[M_2],[M_3],[E_q\rtimes C_{q^2-1}],[(C_{q+1}\times C_{q+1})\rtimes C_2]\}$.
Thus, inductively, $\lambda(H)=0$.
\end{itemize}

Let $H\leq C_{q+1}=Z(M_2(Q))$ for some $Q$ and $K$ a subgroup of $G$ properly containing $H$.
As shown above, $\lambda(K)\ne0$ if and only if $[K]\in\{[G],[M_1],[M_2],[M_3],[E_q\rtimes C_{q^2-1}],[(C_{q+1}\times C_{q+1})\rtimes C_2]\}$.
Thus $\lambda(Z(M_2(Q)))=0$ and, inductively, $\lambda(H)=0$.

Let $H=Sym(3)=\langle\alpha\rangle\rtimes\langle\beta\rangle$ with $o(\alpha)=3$ and $o(\beta)=2$.
Let $P\in\PG(2,q^2)\setminus\cH_q$ and $Q,R\in\cH_q(\fqs)$ be the fixed point of $\alpha$, so that $\beta$ fixes $P$ and interchanges $Q$ and $R$.
This implies $[H]\leq[M_2]$.
By Lemma \ref{S3fissatriangoli}, $[H]\leq[M_3]$.
From the computations above and Lagrange's theorem, no class $[K]$ with $K\leq G$ other than $[G]$, $[M_2]$ and $[M_3]$ satisfies $[H]\leq[K]$ and $\lambda(H)\ne0$.
Thus, $\lambda(H)=1$.

Let $H=C_3$. By Lagrange's theorem and Proposition \ref{unaclassediconiugio}, $H<K\leq G$ and $\lambda(K)\ne0$ if and only if $[K]\in\{[G],[M_1],[M_2],[M_3],[M_4],[E_q\rtimes C_{q^2-1}],[Sym(3)]\}$. Thus, $\lambda(H)=1$.

Let $H=C_2$. By Lagrange's theorem and Proposition \ref{unaclassediconiugio}, $H<K\leq G$ and $\lambda(K)\ne0$ if and only if $[K]\in\{[G],[M_1],[M_2],[M_3],[E_q\rtimes C_{q^2-1}],[(C_{q+1}\times C_{q+1})\rtimes C_2],[Sym(3)]\}$. Thus, $\lambda(H)=-1$.

Let $H=\{1\}$.
Collecting all the classes $[K]$ with $\lambda(K)\ne0$, we have by direct computation $\lambda(H)=0$.
\end{proof}

\section{Determination of $\chi(\Delta(L_p\setminus\{1\}))$ for any prime $p$}\label{sec:euler}

Let $n>0$, $q=2^{2^n}$, $G=\PSU(3,q)$.
If $p$ is a prime number, we denote by $L_p$ the poset of $p$-subgroups of $G$ ordered by inclusion, by $L_p\setminus\{1\}$ its subposet of proper $p$-subgroups of $G$, and by $\Delta(L_p\setminus\{1\})$ the order complex of $L_p\setminus\{1\}$.
In this section we determine the Euler characteristic $\chi(\Delta(L_p\setminus\{1\}))$ of $\Delta(L_p\setminus\{1\})$ for any prime $p$, using Equation \eqref{caratteristica} and Lemma \ref{pgruppi}. The results are stated in Theorem \ref{caratteristicadieulero} and in Table \ref{tabella2}.

\begin{theorem}\label{caratteristicadieulero}
For any prime number $p$ one of the following cases holds:
\begin{itemize}
\item $p\nmid|G|$ and $\chi(\Delta(L_p\setminus\{1\}))=0$;
\item $p=2$ and $\chi(\Delta(L_2\setminus\{1\}))=q^3+1$;
\item $p\mid(q+1)$ and $\chi(\Delta(L_p\setminus\{1\}))=-\frac{q^6-2q^5-q^4+2q^3-3q^2}{3}$;
\item $p\mid(q-1)$ and $\chi(\Delta(L_p\setminus\{1\}))=-\frac{q^6+q^3}{2}$;
\item $p\mid(q^2-q+1)$ and $\chi(\Delta(L_p\setminus\{1\}))=-\frac{q^6+q^5-q^4-q^3}{3}$.
\end{itemize}
\end{theorem}

\begin{proof}
Since $|G|=q^3(q+1)^2(q-1)(q^2-q+1)$, $q$ is even, and $3\mid(q-1)$, the cases $p\nmid|G|$, $p=2$, $p\mid(q+1)$, $p\mid(q-1)$, and $p\mid(q^2-q+1)$ are exhaustive and pairwise incompatible.
We denote by $S_p$ a Sylow $p$-subgroup of $G$.
\begin{itemize}
\item 
Let $p\nmid|G|$. Then $\Delta(L_p\setminus\{1\})=\emptyset$, and hence $\chi(\Delta(L_p\setminus\{1\}))=\chi(\emptyset)=0$.

\item Let $p=2$.
The group $G$ has $q^3+1$ Sylow $2$-subgroups, and two of them intersect trivially; see \cite[Theorem 11.133]{HKT}. Any nontrivial element $\sigma$ of $S_2$ fixes exactly one point $P$ on $\cH_q(\fqs)$ which is the same for any $\sigma\in S_2$; $S_2$ is uniquely determined among the Sylow $2$-subgroups of $G$ by $P$. Hence, Equation \eqref{caratteristica} reads
$$ \chi(\Delta(L_2\setminus\{1\}))=-(q^3+1)\sum_{H\in L_2\setminus\{1\},\; H(P)=P} \mu_{L_2}(\{1\},H), $$
where $P$ is any fixed point of $\cH_q(\fqs)$. By Lemma \ref{pgruppi}, we only consider those $2$-groups in $M_1(P)$ which are elementary abelian.
Then we consider all nontrivial subgroups $H$ of an elementary abelian $2$-group $E_q$ of order $q$. For any such group $H=E_{2^r}$ of order $2^r$, with $1\leq r\leq 2^n$, we have $\mu_{L_2}(\{1\},H)=(-1)^r \cdot 2^{\binom{r}{2}}$ by Lemma \ref{pgruppi}. Thus,
$$ \chi(\Delta(L_2\setminus\{1\})) = -(q^3+1)\sum_{r=1}^{2^n}(-1)^r \, 2^{\binom{r}{2}}\,\binom{2^n}{r}_2 $$
where the Gaussian coefficient $\binom{2^n}{r}_2$ counts the subgroups of $E_q$ of order $2^r$. Using the property
$$\binom{2^n}{r}_2 = \binom{2^n-1}{r-1}_2 + 2^r \binom{2^n-1}{r}_2$$
we obtain
$$ \sum_{r=1}^{2^n}(-1)^r \, 2^{\binom{r}{2}}\,\binom{2^n}{r}_2 = \sum_{r=1}^{2^n}(-1)^r \, 2^{\binom{r}{2}}\,\binom{2^n-1}{r-1}_2 \,+\, \sum_{r=1}^{2^n}(-1)^r \, 2^{\binom{r}{2}+r}\,\binom{2^n-1}{r}_2 $$
$$ = \sum_{r=0}^{2^n-1}(-1)^{r+1} \, 2^{\binom{r+1}{2}}\,\binom{2^n-1}{r}_2 \,+\, \sum_{r=1}^{2^n}(-1)^r \, 2^{\binom{r+1}{2}}\,\binom{2^n-1}{r}_2$$
$$ = (-1)^0\,2^{\binom{1}{2}}\,\binom{2^n-1}{0}_2 \,+\, (-1)^{2^n}\,2^{\binom{2^n+1}{2}}\,\binom{2^n-1}{2^n}_2 =-1. $$
Thus, $\chi(\Delta(L_2\setminus\{1\}))=q^3+1$.

\item
Let $p\mid(q+1)$.
Then $S_p\leq C_{q+1}\times C_{q+1}$, and hence $S_p\cong C_{p^s}\times C_{p^s}$, where $p^s\mid(q+1)$ and $p^{s+1}\nmid(q+1)$.
Let $H$ be a subgroup of $S_p$. By Lemma \ref{pgruppi}, $\mu_{L_p}(\{1\},H)\ne0$ only if $H$ is elementary abelian of order $p$ or $p^2$; in this cases, $\mu_{L_p}(\{1\},C_p)=-1$ and $\mu_{L_p}(\{1\},C_p\times C_p)=r$.
Now we count the number of elementary abelian subgroups of order $p$ or $p^2$ in $G$.
\begin{itemize}
\item
A subgroup $E_{p^2}$ of $G$ of type $C_{p}\times C_p$ is uniquely determined by the maximal subgroup $M_3(T)$ such that $E_{p^2}$ is the Sylow $p$-subgroup of $M_3(T)$.
Hence, $G$ contains exactly $[G:N_G(M_3(T))]=\frac{q^3(q^2-q+1)(q-1)}{6}$ elementary abelian subgroups of order $p^2$.
\item
A subgroup $C_p$ made by homologies is uniquely determined by its center $P\in\PG(2,q^2)\setminus\cH_q$ of homology, because the group of homology with center $P$ is cyclic.
Hence, $G$ contains exactly $|\PG(2,q^2)\setminus\cH_q|=q^2(q^2-q+1)$ cyclic subgroups of order $p$ made by homologies.
\item
A subgroup $C_p$ which is not made by homologies is made by elements of type (B1), and fixes pointwise a unique self-polar triangle $T$.
The Sylow $p$-subgroup $C_{p}\times C_p$ of $M_3(T)$ contains exactly $3$ subgroups $C_{p}$ made by homologies, namely the groups of homologies with center one of the vertexes of $T$.
Since $C_p\times C_p$ contains $p+1$ subgroups $C_p$ altogether, $C_p\times C_p$ contains exactly $p-2$ subgroups $C_p$ not made by homologies.
Thus, the number of subgroups $C_p$ of $G$ not made by homologies is $(p-2)\cdot [G:N_G(M_3(T))]=\frac{q^3(q^2-q+1)(q-1)(p-2)}{6}$.
\end{itemize}
Thus, by direct computation,
\begin{small}
$$ \chi(\Delta(L_p\setminus\{1\}))= -\left\{ \frac{q^3(q^2-q+1)(q-1)(p-2)}{6}\cdot r + \left[q^2(q^2-q+1)+\frac{q^3(q^2-q+1)(q-1)(p-2)}{6}\right]\cdot(-1) \right\}$$
$$ = -\frac{q^6-2q^5-q^4+2q^3-3q^2}{3}. $$
\end{small}

\item Let $p\mid(q-1)$.
By Lemma \ref{classificazione}, $S_p$ is a subgroup of the cyclic group $C_{q^2-1}$ fixing two points $P,Q$ on $\cH_q(\fqs)$;
then a proper $p$-subgroup $H$ of $G$ satisfies $\mu_{L_p}(\{1\})\ne0$ if and only if $H$ has order $p$; in this case, $\mu_{L_p}(\{1\},H)=-1$.
Also, by Lemma \ref{classificazione}, any two Sylow $p$-subgroups of $G$ have trivial intersection. Then the number of subgroups $C_p$ of $G$ is equal to the number $\binom{q^3+1}{1}$ of couples of points in $\cH_q(\fqs)$; equivalently, this number is equal $[G:N_G(C_{q^2})]$, where $|N_G(C_{q^2-1})|=2(q^2-1)$ by Proposition \ref{normalizzante}.
Thus, $\chi(\Delta(L_p\setminus\{1\}))=-\frac{q^6+q^3}{2}$.

\item Let $p\mid(q^2-q+1)$.
Then $S_p\leq C_{q^2-q+1}$, and hence a proper $p$-subgroup $H$ of $G$ satisfies $\mu_{L_p}(\{1\},H)\ne0$ if and only if $H$ has order $p$; in this case, $\mu_{L_p}(\{1\},H)=-1$.
The number of subgroups $C_p$ of $G$ is equal to the number of subgroups $C_{q^2-q+1}$, and hence to the number $[G:N_G(M_4(\tilde T))]=\frac{q^3(q+1)^2(q-1)}{3}$ of maximal subgroups of type $M_4(\tilde T)$ in $G$.
Thus, $\chi(\Delta(L_p\setminus\{1\}))=-\frac{q^3(q+1)^2(q-1)}{3}=-\frac{q^6+q^5-q^4-q^3}{3}$.
\end{itemize}
\end{proof}

\section{Acknowledgements}

This research was partially supported by Ministry for Education, University and Research of Italy (MIUR) and by the Italian National Group for Algebraic and Geometric Structures and their Applications (GNSAGA - INdAM).
The author would like to thank Francesca Dalla Volta who introduced him to M\"obius functions on groups and supported him with a number of helpful conversations. Many thanks are also due to Martino Borello for his useful comments, and to Emilio Pierro for pointing out a mistake in a previous version of this paper.

\end{document}